\documentclass{article}

\usepackage{amsmath,amsthm,amssymb,a4wide,hyperref,graphicx}
\usepackage[all]{xy}
\usepackage{times}

\newcommand{\M}{\mathcal{M}}

\renewcommand{\phi}{\varphi}

\newcommand{\BP}{\textbf{P}}

\newcommand{\lr}[1]{\langle #1 \rangle}

\renewcommand{\iff}{\Longleftrightarrow}
\newcommand{\lra}{\leftrightarrow}

\newcommand{\weg}[1]{}

\newcommand{\bis}{\mathrel{\mathchoice%
{\raisebox{.3ex}{$\,
  \underline{\makebox[.7em]{$\leftrightarrow$}}\,$}}%
{\raisebox{.3ex}{$\,
  \underline{\makebox[.7em]{$\leftrightarrow$}}\,$}}%
{\raisebox{.2ex}{$\,
  \underline{\makebox[.5em]{\scriptsize$\leftrightarrow$}}\,$}}%
{\raisebox{.2ex}{$\,
  \underline{\makebox[.5em]{\scriptsize$\leftrightarrow$}}\,$}}}}

\theoremstyle{definition}
\newtheorem{theorem}{Theorem}[section]
\newtheorem{lemma}[theorem]{Lemma}
\newtheorem{definition}[theorem]{Definition}

\newtheorem{fact}[theorem]{Fact}

\newtheorem{conjecture}[theorem]{Conjecture}
\newtheorem{proposition}[theorem]{Proposition}

\newtheorem{corollary}[theorem]{Corollary}

\title{Notes on neighborhood semantics for logics of unknown truths and false beliefs}
\author{Jie Fan\\
\small School of Humanities, University of Chinese Academy of Sciences  \\
\small \texttt{jiefan@ucas.ac.cn}}
\date{}

\begin{document}
\maketitle

\begin{abstract}
In this article, we study logics of unknown truths and false beliefs under neighborhood semantics. We compare the relative expressivity of the two logics. It turns out that they are incomparable over various classes of neighborhood models, and the combination of the two logics are equally expressive as standard modal logic over any class of neighborhood models. We propose morphisms for each logic, which can help us explore the frame definability problem, show a general soundness and completeness result, and generalize some results in the literature. We axiomatize the two logics over various classes of neighborhood frames. Last but not least, we extend the results to the case of public announcements, which has good applications to Moore sentences and some others.
\end{abstract}

\noindent Keywords: unknown truths, false beliefs, accident, neighborhood semantics, morphisms, axiomatizations, expressivity, frame definability, intersection semantics

\section{Introduction}

This paper studies logics of unknown truths and false beliefs under neighborhood semantics. Intuitively, if $p$ is true but you do not know that $p$, then you have an {\em unknown truth} that $p$; if $p$ is false but you believe that $p$, then you have a {\em false belief} that $p$, or you {\em are wrong about} $p$.


The notion of unknown truths is important in philosophy and formal epistemology. For instance, it is related to Verificationism, or `verification thesis'~\cite{jfak.fitch:2004}. Verificationism says that all truths can be known. However, from the thesis, the unknown truth of $p$, formalized $p\land \neg Kp$, gives us a consequence that all truths are actually known. In other words, the notion gives rise to a well-known counterexample to Verificationism. This is the so-called Fitch's `paradox of knowability'~\cite{fitch:1963}.\footnote{For an excellent survey on Fitch's paradox of knowability, we refer to~\cite{sep-fitch-paradox}.} To take another example: it gives rise to an important type of Moore sentences, which is essential to Moore's paradox, which says that one cannot claim the paradoxical sentence ``$p$ but I do not know it''~\cite{moore:1942,hintikka:1962}. It is known that such a Moore sentence is unsuccessful and self-refuting (see, e.g.~\cite{hollidayetal:2010,jfak.book:2011,hvdetal.del:2007}).\footnote{To say a formula $\phi$ is {\em successful}, if it still holds after being announced, in symbol $\vDash[\phi]\phi$. Otherwise, we say this formula is {\em unsuccessful}. Moreover, to say $\phi$ is {\em self-refuting}, if its negation always holds after being announced, in symbol $\vDash[\phi]\neg\phi$.}


In addition to the axiomatization for the logic of unknown truths on topological semantics~\cite{steinsvold:2008}, there has been various work on the metaphysical counterpart of unknown truths --- accidental truths, or simply, `accident'. The notion of accidental truths traces back at least to Leibniz, in the name of `v\'{e}rit\'{e}s de fait' (factual truths), see e.g.~\cite{AD:1989A,Heinemann:1948}. This notion is related to problem of {\em future contingents}, which is formalized by a negative form of accident~\cite{Aristotle:1941}. Moreover, it is applied to reconstruct G\"{o}del's ontological argument (e.g.~\cite{Small:2001}), and also to provide an additional partial verification of the Boxdot Conjecture raised in~\cite{FH:2009} (also see~\cite{Steinsvold:2011}).


The logical investigation on the notion of accidental truths is initiated by Marcos, who axiomatizes a minimal logic of accident under relational semantics in~\cite{Marcos:2005}, to differentiate `accident' from `contingency'.\footnote{As for a recent survey on (non)contingency logic, we refer to~\cite{fan2019symmetric}.} The axiomatization is then simplified and its various extensions are presented in~\cite{Steinsvold:2008a}. Symmetric accident logic is axiomatized in~\cite{Fan:2015accident}, and Euclidean accident logic is explored in~\cite{BalbianiFan:2017}.\footnote{In fact, \cite{BalbianiFan:2017} gave a complete axiomatization for strong noncontingency logic $\mathcal{L}(\blacktriangle)$ over the class of Euclidean frames, thereby answering an open question posed in~\cite{Fan:2019strongnoncontingency}. However, since as shown in~\cite{Fan:2019strongnoncontingency}, $\mathcal{L}(\blacktriangle)$ is equally expressive as accident logic over the class of arbitrary models, thus one can translate the axiomatization of Euclidean strong noncontingency logic into an axiomatization of Euclidean accident logic.} Some quite general soundness and completeness results can be found in~\cite{GilbertVenturi:2016}. Some relative expressivity results are obtained in~\cite{Marcos:2005,Fan:2015accident}.

In comparison, the notion of false beliefs is popular in the area of cognitive science, see e.g.~\cite{perner1987three,wellman2001meta}. For technical reasons, \cite{Steinsvold:falsebelief} proposes a logic that has the operator $W$ as a sole modality. There, $W\phi$ is read ``the agent is wrong about $\phi$'', and being wrong about $\phi$ means believing $\phi$ though $\phi$ is false. Complete axiomatizations of the minimal logic of false belief and its various extensions are given, and some results of frame definability are presented.





 However, all this work are based on relational semantics. As the logics of unknown truths and false beliefs are non-normal (due to the non-normality of their modalities), it is then natural and interesting to investigate them from the perspective of neighborhood semantics.

Neighborhood semantics is independently proposed by Scott and Montague in 1970~\cite{Scott:1970,Montague:1970}. Since it is introduced, neighborhood semantics has become a standard semantics for investigating non-normal modal logics~\cite{Chellas1980}. Partly inspired by~\cite{FanvD:neighborhood}, the authors of~\cite{GilbertVenturi:2017} proposes neighborhood semantics for logics of unknown truths and false beliefs. According to the semantics, ``it is an unknown truth that $\phi$'' is interpreted as ``$\phi$ is true and the proposition expressed by $\phi$ is not a neighborhood of the evaluated state'', and ``it is a false belief that $\phi$'' as ``$\phi$ is false and the proposition expressed by $\phi$ is a neighborhood of the evaluated state''. Beyond some invariance and negative results, a minimal logic of unknown truths under relational semantics, denoted ${\bf B_K}$ there, is shown to be sound and complete with respect to the class of filters, and a minimal logic of false beliefs under relational semantics, denoted ${\bf A_K}$ therein, is shown to be sound and complete with respect to the class of neighborhood frames that are closed under binary intersections and are negatively supplemented.

In this paper, in addition to explore the relative expressivity of logics of unknown truths and false beliefs over various classes of neighborhood models, we also axiomatize logics of unknown truths and false beliefs over various neighborhood frames. By defining notions of $\bullet$-morphisms and $W$-morphisms, we obtain good applications to, e.g. frame (un)definability, a general soundness and completeness result, and some results that generalize those in~\cite{GilbertVenturi:2017} in a relative easy way. Moreover, we extend the results to the case of public announcements: by adopting the intersection semantics in the literature (which is a kind of neighborhood semantics for public announcements), we find suitable reduction axioms and thus complete proof systems, which, again, gives us good applications to some interesting questions. For instance, are Moore sentences self-refuting? How about the negation of Moore sentences? Are false beliefs of a fact successful formulas? Other natural questions also result, for instance, are all unknown truths themselves unknown truths? Are all false beliefs themselves are false beliefs?

As we will show in a proof-theoretical way, interestingly, under fairly weak assumption (namely, monotonicity), one's false belief of a fact cannot be removed even after being told: if you have a false belief, then after someone tells you this, you still have the false belief. In other words, false beliefs of facts are all successful formulas. Different from the case in relational semantics, under neighborhood semantics, Moore sentences are {\em not} self-refuting in general. But the negation of Moore sentences are successful in the presence of monotonicity. Also, all unknown truths themselves unknown truths, but not all false beliefs themselves are false beliefs, indeed, {\em none} of false beliefs themselves are false beliefs.

The reminder of the paper is organized as follows. After reviewing the languages and their neighborhood semantics and some common neighborhood properties (Sec.~\ref{sec.syntaxandsemantics}), we compare the relative expressivity of the languages in Sec.~\ref{sec.expressivity}. Sec.~\ref{sec.morphisms} proposes notions of $\bullet$-morphisms and $W$-morphisms and exploit their applications. Sec.~\ref{sec.axiomatizations} axiomatizes the logics over various classes of neighborhood frames, which include a general soundness and completeness result shown via the notion of $W$-morphisms. Sec.~\ref{sec.publicannouncements} extends the previous results to the case of public announcements, where by using intersection semantics for public announcements, we find suitable reduction axioms and complete axiomatizations, which gives us good applications to Moore sentences and some others. We conclude with some future work in Sec.~\ref{sec.conclusion}.
\section{Syntax and Semantics}\label{sec.syntaxandsemantics}

Throughout this paper, we fix a nonempty set of propositional variables $\BP$ and $p\in\BP$.

\begin{definition} The languages involved in the current paper include the following. 
\[
\begin{array}{ll}
\mathcal{L}(\bullet)& \phi::=p\mid \neg\phi\mid \phi\land\phi\mid \bullet\phi\\
\mathcal{L}(W) & \phi::=p\mid \neg\phi\mid \phi\land\phi\mid W\phi\\
\mathcal{L}(\bullet,W)&\phi::=p\mid \neg\phi\mid \phi\land\phi\mid \bullet\phi\mid W\phi\\
\mathcal{L}(\Box)&\phi::=p\mid \neg\phi\mid \phi\land\phi\mid \Box\phi\\
\end{array}
\]

$\mathcal{L}(\bullet)$ is the {\em language of the logic of unknown truths}, $\mathcal{L}(W)$ is the {\em language of the logic of false beliefs}, $\mathcal{L}(\bullet,W)$ is the {\em language of the logic of unknown truths and false beliefs}, and $\mathcal{L}(\Box)$ is the {\em language of epistemic/doxastic logic}.
\end{definition}

Intuitively, $\bullet\phi$ is read ``it is an {\em unknown truth} that $\phi$'', that is, ``$\phi$ is true but unknown'', $W\phi$ is read ``the agent {\em is wrong about} $\phi$'', or ``it is a {\em false belief} that $\phi$'', that is, ``$\phi$ is false but believed'', and $\Box\phi$ is read ``it is known/believed that $\phi$''. Other connectives are defined as usual; in particular, $\circ\phi$ is abbreviated as $\neg\bullet\phi$, read ``it is known that $\phi$ once it is the case that $\phi$''. In a philosophical context, $\bullet\phi$, $\circ\phi$, and $\Box\phi$ are read ``it is accident (or accidentally true) that $\phi$'', ``it is essential that $\phi$'', and ``it is necessary that $\phi$'', respectively.

All the above-mentioned languages are interpreted over neighborhood models.

\begin{definition}
A (neighborhood) model is a triple $\M=\lr{S,N,V}$ such that, $S$ is a nonempty set of states (or called `possible worlds'), $N$ is a neighborhood function from $S$ to $\mathcal{P}(\mathcal{P}(S))$, and $V$ is a valuation function. Intuitively, $X\in N(s)$ means that $X$ is a neighborhood of $s$. For any neighborhood model $\M$ and state $s$ in $\M$, $(\M,s)$ is called a {\em pointed (neighborhood) model}. Without considering the valuation function, we obtain a (neighborhood) frame.
\end{definition}

Given a neighborhood model $\M=\lr{S,N,V}$ and a state $s\in S$, the semantics of the aforementioned languages is defined inductively as follows.
\[
\begin{array}{|lll|}
\hline
\M,s\vDash p&\iff & s\in V(p)\\
\M,s\vDash \neg\phi &\iff &\M,s\nvDash\phi\\
\M,s\vDash \phi\land\psi &\iff& \M,s\vDash\phi\text{ and }\M,s\vDash\psi\\
\M,s\vDash\bullet\phi&\iff& s\in\phi^{\M}\text{ and }\phi^{\M}\notin N(s)\\
\M,s\vDash W\phi&\iff &\phi^\M\in N(s)\text{ and }s\notin \phi^\M\\
\M,s\vDash\Box\phi&\iff& \phi^{\M}\in N(s)\\
\hline
\end{array}
\]
Where $\phi^{\M}=\{s\in\M\mid \M,s\vDash\phi\}$.

It is easily computed that
\[
\begin{array}{lll}
\M,s\vDash\circ\phi&\iff& s\in \phi^\M\text{ implies }\phi^{\M}\in N(s).\\
\end{array}
\]

\weg{Recall that the standard knowledge/necessity operator is interpreted in the following.
\[
\begin{array}{lll}
\M,s\vDash\Box\phi&\iff& \phi^{\M}\in N(s)\\
\end{array}
\]}

Thus one may easily verify that $\vDash\bullet\phi\lra (\phi\land\neg\Box\phi)$, $\vDash W\phi\lra \Box\phi\land\neg\phi$, $\vDash\circ\phi\lra (\phi\to\Box\phi)$, which conform to the previous readings of $\bullet\phi$, $W\phi$, $\circ\phi$, respectively. This indicates that the modalities $\bullet$, $W$, $\circ$ are all definable in the standard modal logic $\mathcal{L}(\Box)$, and therefore $\mathcal{L}(\Box)$ is at least as expressive as $\mathcal{L}(\bullet)$ and also $\mathcal{L}(W)$ over any class of neighborhood models.

The neighborhood properties which we mainly focus on in this paper include the following.
\begin{definition}[Neighborhood properties] Let $\mathcal{F}=\lr{S,N}$ be a neighborhood frame, and $\M$ be a neighborhood model based on $\mathcal{F}$. For each $s\in S$ and $X,Y\subseteq S$:
\begin{enumerate}
\item[$(m)$] $N(s)$ is {\em supplemented}, or {\em closed under supersets}, if $X\in N(s)$ and $X\subseteq Y$ implies $Y\in N(s)$. In this case, we also say that $N(s)$ is {\em monotone}.
\item[$(c)$] $N(s)$ is {\em closed under (binary) intersections}, if $X\in N(s)$ and $Y\in N(s)$ implies $X\cap Y\in N(s)$.
\item[$(n)$] $N(s)$ {\em contains the unit}, if $S\in N(s)$.
\item[$(r)$] $N(s)$ {\em contains its core}, if $\bigcap N(s)\in N(s)$.
\weg{\item[(u)] $X\notin N(s)$ implies $s\in X$.
\item[(f)] $X\in N(s)$ implies $s\notin X$.}
\end{enumerate}

The function $N$ possesses such a property, if $N(s)$ has the property for all $s\in S$; $\mathcal{F}$ has a property, if $N$ has. Frame $\mathcal{F}$ is a filter, if $\mathcal{F}$ has $(m)$, $(c)$ and $(n)$; $\mathcal{F}$ is augmented, if $\mathcal{F}$ has $(m)$ and $(r)$. Model $\M$ has a property, if $\mathcal{F}$ has such a property.
\end{definition}

It is known that every augmented model is a filter, but not vice versa (see e.g.~\cite{Chellas1980}).
\section{Expressivity}\label{sec.expressivity}

This part compares the relative expressivity of $\mathcal{L}(\bullet)$ and $\mathcal{L}(W)$. To begin with, we give the definition of expressivity.

\begin{definition}[Expressivity] Let $\mathcal{L}_1$ and $\mathcal{L}_2$ be two logical languages that are interpreted in the same class $\mathbb{M}$ of models,
\begin{itemize}
\item $\mathcal{L}_2$ is {\em at least as expressive as} $\mathcal{L}_1$, notation: $\mathcal{L}_1\preceq\mathcal{L}_2$, if for each formula $\phi$ in $\mathcal{L}_1$, there exists a formula $\psi$ in $\mathcal{L}_2$ such that for each model $\M$ in $\mathbb{M}$, for each state $s$ in $\M$, we have that $\M,s\vDash\phi$ iff $\M,s\vDash\psi$.
\item $\mathcal{L}_1$ is {\em less expressive than} $\mathcal{L}_2$, notation: $\mathcal{L}_1\prec\mathcal{L}_2$, if $\mathcal{L}_1\preceq\mathcal{L}_2$ and $\mathcal{L}_2\not\preceq\mathcal{L}_1$.
\item $\mathcal{L}_1$ and $\mathcal{L}_2$ are {\em equally expressive}, if $\mathcal{L}_1\preceq\mathcal{L}_2$ and $\mathcal{L}_2\preceq \mathcal{L}_1$.
\item $\mathcal{L}_1$ and $\mathcal{L}_2$ are {\em incomparable (in expressivity)}, if $\mathcal{L}_1\not\preceq\mathcal{L}_2$ and $\mathcal{L}_2\not\preceq \mathcal{L}_1$.
\end{itemize}
\end{definition}

The following two propositions state that the languages $\mathcal{L}(\bullet)$ and $\mathcal{L}(W)$ are incomparable over any model classes with the above neighborhood properties.
\begin{proposition}\label{prop.exp-lcirc-lw}
On the class of all models, the $(m)$-models, the $(c)$-models, the $(n)$-models, the $(r)$-models, $\mathcal{L}(\bullet)$ is not at least as expressive as $\mathcal{L}(W)$.
\end{proposition}

\begin{proof}
Consider the following models, where the only difference is $N'(s)=N(s)\cup\{\{t\}\}$, and an arrow from a state $x$ to a set $X$ means that $X$ is a neighborhood of $x$:
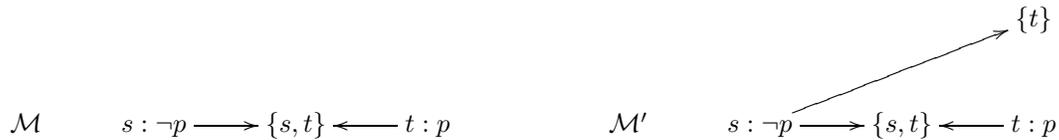
$$
\xymatrix{&&&&&&&&\{t\}\\
\M&s:\neg p\ar[r]&\{s,t\}&t:p\ar[l]&&\M'&s:\neg p\ar[r]\ar[urr]&\{s,t\}&t:p\ar[l]}
$$

It may be easily checked that both $\M$ and $\M'$ have $(m)$, $(c)$, $(n)$ and $(r)$.

Moreover, $(\M,s)$ and $(\M',s)$ can be distinguished by an $\mathcal{L}(W)$-formula: on the one hand, as $p^\M=\{t\}\notin N(s)$, we have $\M,s\nvDash Wp$; on the other hand, since $\M',s\nvDash p$ and $p^{\M'}=\{t\}\in N'(s)$, we infer that $\M,s\vDash Wp$.

However, these two pointed models cannot be distinguished by any $\mathcal{L}(\bullet)$-formulas. For this, we show a stronger result that for all $\phi\in\mathcal{L}(\bullet)$, for all $x\in S$, $\M,x\vDash\phi$ iff $\M',x\vDash\phi$, that is, $\phi^\M=\phi^{\M'}$. As the two models differs only in the neighborhood of $s$, it suffices to show that $\M,s\vDash\phi$ iff $\M',s\vDash\phi$, that is, $s\in \phi^\M$ iff $s\in \phi^{\M'}$. The proof goes with induction on $\phi$, where the only case to treat is $\bullet\phi$.

To begin with, suppose that $\M,s\vDash\bullet\phi$, then $s\in\phi^\M$ and $\phi^\M\notin N(s)$. By induction hypothesis, $s\in \phi^{\M'}$ and $\phi^{\M'}\notin N(s)$. Since $s\in \phi^{\M'}$, it must be the case that $\phi^{\M'}\neq \{t\}$, that is, $\phi^{\M'}\notin \{\{t\}\}$, and thus $\phi^{\M'}\notin N(s)\cup\{\{t\}\}=N'(s)$. Therefore, $\M',s\vDash\bullet\phi$.

Conversely, assume that $\M',s\vDash\bullet\phi$, then $s\in \phi^{\M'}$ and $\phi^{\M'}\notin N'(s)$. As $N(s)\subseteq N'(s)$, by induction hypothesis, we infer that $s\in \phi^\M$ and $\phi^{\M}\notin N(s)$. Therefore, $\M,s\vDash\bullet\phi$.

Therefore, $\mathcal{L}(W)\not\preceq\mathcal{L}(\bullet)$.
\end{proof}

\begin{proposition}\label{prop.exp-lw-lcirc}
On the class of all models, the $(m)$-models, the $(c)$-models,  $(n)$-models, the $(r)$-models, $\mathcal{L}(W)$ is not at least as expressive as $\mathcal{L}(\bullet)$.
\end{proposition}

\begin{proof}
Consider the following models, where the only difference is that $N'(s)=N(s)\cup\{\{s\}\}$:
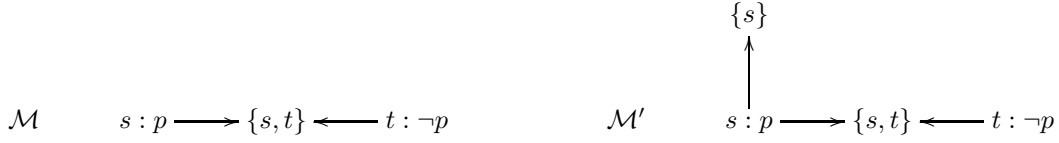
$$
\xymatrix{&&&&&&\{s\}&&\\
\M&s:p\ar[r]&\{s,t\}&t:\neg p\ar[l]&&\M'&s:p\ar[r]\ar[u]&\{s,t\}&t:\neg p\ar[l]}
$$

One may check that $\M$ and $\M'$ both have $(m)$, $(c)$, $(n)$ and $(r)$.

One the one hand, $(\M,s)$ and $(\M',s)$ can be distinguished by an $\mathcal{L}(\bullet)$-formula, just noticing that $\M,s\vDash\bullet p$ (as $\M,s\vDash p$ but $p^\M=\{s\}\notin N(s)$) and $\M',s\nvDash\bullet p$ (since $p^{\M'}=\{s\}\in N'(s)$).

On the other hand, $(\M,s)$ and $(\M',s)$ cannot be distinguished by any $\mathcal{L}(W)$-formulas. For this, we prove a stronger result that for all $\phi\in\mathcal{L}(W)$, for all $x\in S$, $\M,x\vDash\phi$ iff $\M',x\vDash\phi$, that is, $\phi^\M=\phi^{\M'}$. As the two models differs only in the neighborhood of $s$, it is sufficient to demonstrate that $\M,s\vDash\phi$ iff $\M',s\vDash\phi$. The proof continues with induction on $\phi$, in which the only case to fix is $W\phi$.

First, suppose that $\M,s\vDash W\phi$, then $\phi^\M\in N(s)$ and $s\notin \phi^\M$. Since $N(s)\subseteq N'(s)$, by induction hypothesis, we can obtain that $\phi^{\M'}\in N'(s)$ and $s\notin \phi^{\M'}$, and thus $\M',s\vDash W\phi$.

For the other direction, assume that $\M',s\vDash W\phi$, then $\phi^{\M'}\in N'(s)$ and $s\notin \phi^{\M'}$. As $s\notin \phi^{\M'}$, it must be the case that $\phi^{\M'}\neq \{s\}$, that is, $\phi^{\M'}\notin \{\{s\}\}$. Thus $\phi^{\M'}\in N(s)$. By induction hypothesis, we infer that $\phi^\M\in N(s)$ and $s\notin\phi^\M$, therefore $\M,s\vDash W\phi$.

Therefore, $\mathcal{L}(\bullet)\not\preceq \mathcal{L}(W)$.
\end{proof}

The following result follows immediately from Prop.~\ref{prop.exp-lcirc-lw} and Prop.~\ref{prop.exp-lw-lcirc}.
\begin{corollary}
On the class of all models, the $(m)$-models, the $(c)$-models, the $(n)$-models, the $(r)$-models, $\mathcal{L}(\bullet)$ and $\mathcal{L}(W)$ are incomparable, and thus both logics are less expressive than $\mathcal{L}(\bullet,W)$.
\end{corollary}

The result below states that $\mathcal{L}(\bullet,W)$ is equally expressive as $\mathcal{L}(\Box)$ over any class of neighborhood models. This extends the result in~\cite{Fan:2019sane}, where it is shown that the two logics are equally expressive over any class of relational models.
\begin{proposition}
$\mathcal{L}(\bullet,W)$ is equally expressive as $\mathcal{L}(\Box)$ on any class of neighborhood models.
\end{proposition}

\begin{proof}
Since $\vDash\bullet\phi\lra \phi\land\neg\Box\phi$ and $\vDash W\phi\lra \Box\phi\land\neg\phi$, we have $\mathcal{L}(\bullet,W)\preceq \mathcal{L}(\Box)$.

Moreover, we demonstrate that $\vDash \Box\phi\lra W\phi\vee(\circ\phi\land\phi)$, as follows. Given any neighborhood model $\M=\lr{S,N,V}$ and $s\in S$, we have the following equivalences:
\[\begin{array}{ll}
&\M,s\vDash W\phi\vee(\circ\phi\land\phi)\\
\iff &\M,s\vDash W\phi\text{ or }\M,s\vDash \circ\phi\land\phi\\
\iff &(\phi^\M\in N(s)\text{ and }\M,s\nvDash\phi)\text{ or }(\M,s\vDash\circ\phi\text{ and }\M,s\vDash\phi)\\
\iff &(\phi^\M\in N(s)\text{ and }\M,s\nvDash\phi)\text{ or }((\M,s\vDash\phi\text{ implies }\phi^\M\in N(s))\text{ and }\M,s\vDash\phi)\\
\iff &(\phi^\M\in N(s)\text{ and }\M,s\nvDash\phi)\text{ or }(\M,s\vDash\phi\text{ and }\phi^\M\in N(s))\\
\iff &\phi^M\in N(s)\\
\iff &\M,s\vDash\Box\phi.
\end{array}\]

This implies that $\mathcal{L}(\Box)\preceq \mathcal{L}(\bullet,W)$, and therefore $\mathcal{L}(\bullet,W)$ is equally expressive as $\mathcal{L}(\Box)$ on any class of neighborhood models.
\end{proof}

\section{Morphisms and their applications}\label{sec.morphisms}

This section proposes notions of morphisms for $\mathcal{L}(\bullet)$ and $\mathcal{L}(W)$, and some of their applications.

\subsection{$\bullet$-morphisms}

\begin{definition}[$\bullet$-Morphisms]
Let $\M=\lr{S,N,V}$ and $\M'=\lr{S',N',V'}$ be neighborhood models. A function $f:S\to S'$ is a {\em $\bullet$-morphism} from $\M$ to $\M'$, if for all $s\in S$,
\begin{itemize}
\item[(Var)] $s\in V(p)$ iff $f(s)\in V(p)$ for all $p\in \BP$;
\item[($\bullet$-Mor)] for all $X\subseteq S$, $[s\in X\text{ and }X\notin N(s)]\iff [f(s)\in f[X]\text{ and }f[X]\notin N'(f(s))].$
\end{itemize}

We say that $\M'$ is a $\bullet$-morphic image of $\M$, if there is a surjective $\bullet$-morphism from $\M$ to $\M'$.
\end{definition}

\weg{\begin{definition}[$\circ$-Morphisms]
Let $\M=\lr{S,N,V}$ and $\M'=\lr{S',N',V'}$ be neighborhood models. A function $f:S\to S'$ is a {\em $\circ$-morphism} from $\M$ to $\M'$, if for all $s\in S$,
\begin{itemize}
\item[(Var)] $s\in V(p)$ iff $f(s)\in V(p)$ for all $p\in \BP$;
\item[($\circ$-Mor)] for all $X\subseteq S$, $[s\notin X\text{ or }X\in N(s)]\iff [f(s)\notin f[X]\text{ or }f[X]\in N'(f(s))].$
\end{itemize}

We say that $\M'$ is a $\circ$-morphic image of $\M$, if there is a surjective $\circ$-morphism from $\M$ to $\M'$.
\end{definition}}

The following result indicates that the formulas of $\mathcal{L}(\bullet)$ are invariant under $\bullet$-morphisms.

\begin{proposition}\label{prop.circ-morphism}
Let $\M=\lr{S,N,V}$ and $\M'=\lr{S',N',V'}$ be neighborhood models, and let $f$ be a $\bullet$-morphism from $\M$ to $\M'$. Then for all $s\in S$, for all $\phi\in\mathcal{L}(\bullet)$, we have $\M,s\vDash\phi\iff \M',f(s)\vDash\phi$, that is, $f[\phi^\M]=\phi^{\M'}$.
\end{proposition}

\begin{proof}
By induction on $\phi$. The nontrivial case is $\bullet\phi$.

Suppose that $\M,s\vDash\bullet\phi$, to show that $\M',f(s)\vDash\bullet\phi$. By supposition, $s\in \phi^\M$ and $\phi^\M\notin N(s)$. By ($\bullet$-Mor), we have that $f(s)\in f[\phi^\M]$ and $f[\phi^\M]\notin N'(f(s))$. By induction hypothesis, this means that $f(s)\in \phi^{\M'}$ and $\phi^{\M'}\notin N'(f(s))$. Thus $\M',f(s)\vDash\bullet\phi$.

Conversely, assume that $\M',f(s)\vDash\bullet\phi$, to prove that $\M,s\vDash\bullet\phi$. By assumption, $f(s)\in \phi^{\M'}$ and $\phi^{\M'}\notin N'(f(s))$. By induction hypothesis, this entails that $f(s)\in f[\phi^\M]$ and $f[\phi^\M]\notin N'(f(s))$. By ($\bullet$-Mor) again, we obtain that $s\in \phi^\M$ and $\phi^\M\notin N(s)$. Therefore, $\M,s\vDash\bullet\phi$.
\end{proof}

The notion of $\bullet$-morphisms can be applied to the following result in a relative easy way. Note that $\M^{t^+}$ and $\M^{t^-}$ defined on~\cite[p.~254]{GilbertVenturi:2017} are, respectively, the special cases of $\M^{t^+}$ and $\M^{t^-}$ defined below when $\Gamma_w=S_w$. Thus our result below is an extension of~\cite[Thm.~1.10]{GilbertVenturi:2017}.

\weg{\begin{proposition}\label{prop.equivalence-circ}
Let $\M=\lr{S,N,V}$. For each $w\in S$ and $\alpha\in\mathcal{L}(\circ)$, we have
$$\M,w\vDash\alpha\text{ iff }\M^{t^+},w\vDash\alpha$$
and
$$\M,w\vDash\alpha\text{ iff }\M^{t^-},w\vDash\alpha,$$
where $\M^{t^+}=\lr{S,N^{t^+},V}$ and $\M^{t^-}=\lr{S,N^{t^-},V}$, where $N^{t^+}(w)=N(w)\cup \Gamma_w$ and $N^{t^-}(w)=N(w)\backslash \Gamma_w$, in which $\Gamma_w\subseteq S_w=\{X\subseteq S\mid w\notin X\}$.
\end{proposition}

\begin{proof}
By Prop.~\ref{prop.circ-morphism}, it is sufficient to show that $f:S\to S$ such that $f(x)=x$ is a $\circ$-morphism from $\M$ to $\M^{t^+}$, and also a $\circ$-morphism from $\M$ to $\M^{t^-}$.

The condition (Var) is clear. For (Mor), we need to show that

$$[w\notin X\text{ or }X\in N(w)]\iff [w\notin X\text{ or }X\in N^{t^+}(w)]~~~~~~~~~~~~~~~~~~~~~~~~~(1)$$
and
$$[w\notin X\text{ or }X\in N(w)]\iff [w\notin X\text{ or }X\in N^{t^-}(w)]~~~~~~~~~~~~~~~~~~~~~~~~~(2).$$
The ``$\Longrightarrow$'' of (1) and ``$\Longleftarrow$'' of (2) follows directly from the fact that $N^{t^-}(w)\subseteq N(w)\subseteq N^{t^+}(w)$.

Moreover, if $w\in X$, then according to the definition of $S_w$, we have $X\notin S_w$, thus $X\notin \Gamma_w$. This follows that ``$\Longleftarrow$'' of (1) and ``$\Longrightarrow$'' of (2).
\end{proof}}

\begin{proposition}\label{prop.equivalence-circ}
Let $\M=\lr{S,N,V}$. For each $w\in S$ and $\alpha\in\mathcal{L}(\bullet)$, we have
$$\M,w\vDash\alpha\text{ iff }\M^{t^+},w\vDash\alpha$$
and
$$\M,w\vDash\alpha\text{ iff }\M^{t^-},w\vDash\alpha,$$
where $\M^{t^+}=\lr{S,N^{t^+},V}$ and $\M^{t^-}=\lr{S,N^{t^-},V}$, where $N^{t^+}(w)=N(w)\cup \Gamma_w$ and $N^{t^-}(w)=N(w)\backslash \Gamma_w$, in which $\Gamma_w\subseteq S_w=\{X\subseteq S\mid w\notin X\}$.
\end{proposition}

\begin{proof}
By Prop.~\ref{prop.circ-morphism}, it is sufficient to show that $f:S\to S$ such that $f(x)=x$ is a $\bullet$-morphism from $\M$ to $\M^{t^+}$, and also a $\bullet$-morphism from $\M$ to $\M^{t^-}$.

The condition (Var) is clear. For (Mor), we need to show that

$$[w\in X\text{ and }X\notin N(w)]\iff [w\in X\text{ and }X\notin N^{t^+}(w)]~~~~~~~~~~~~~~~~~~~~~~~~~(1)$$
and
$$[w\in X\text{ and }X\notin N(w)]\iff [w\in X\text{ and }X\notin N^{t^-}(w)]~~~~~~~~~~~~~~~~~~~~~~~~~(2).$$
The ``$\Longleftarrow$'' of (1) and ``$\Longrightarrow$'' of (2) follows directly from the fact that $N^{t^-}(w)\subseteq N(w)\subseteq N^{t^+}(w)$.

Moreover, given $w\in X$, according to the definition of $S_w$, we have $X\notin S_w$, thus $X\notin \Gamma_w$. This follows that ``$\Longrightarrow$'' of (1) and ``$\Longleftarrow$'' of (2).
\end{proof}

Note that in the above proposition, as $\Gamma_w$ is defined in terms of $w$, thus given any two points $x,y\in S$, $\Gamma_x$ may be different from $\Gamma_y$. This point will be used frequently in the proofs below.

Now coming back to Prop.~\ref{prop.exp-lcirc-lw}, instead of directly proving that $\mathcal{L}(\bullet)$-formulas cannot distinguish between $(\M,s)$ and $(\M',s)$, we can resort to Prop.~\ref{prop.equivalence-circ}, by just noticing that $\M'=\M^{t^+}$ where $\Gamma_s=\{\{t\}\}$ and $\Gamma_t=\emptyset$.\footnote{Note that $\Gamma_s$ is an arbitrary subset of $S_s$, and $\{t\}\in S_s$ (as $s\notin \{t\}$), we thus can set $\Gamma_s=\{\{t\}\}$. Similar arguments apply for $\Gamma_t$ and other similar definitions of $\Gamma_s$ and $\Gamma_t$ in other situations below.}

With Prop.~\ref{prop.equivalence-circ}, we immediately have the following corollary, which extends the result in~\cite[Coro.~1.11]{GilbertVenturi:2017}.
\begin{corollary}\label{coro.frameequivalent-circ}
Let $\mathcal{F}=\lr{S,N}$, and $\mathcal{F}^{t^+}=\lr{S,N^{t^+}}$ and $\mathcal{F}^{t^-}=\lr{S,N^{t^-}}$ be defined as in Prop.~\ref{prop.equivalence-circ}. Then for all $\phi\in\mathcal{L}(\bullet)$, we have
$$\mathcal{F}\vDash\phi\text{ iff }\mathcal{F}^{t^+}\vDash\phi$$
and
$$\mathcal{F}\vDash\phi\text{ iff }\mathcal{F}^{t^-}\vDash\phi.$$
\end{corollary}

It turns out that this corollary is quite useful in exploring the problem of frame (un)definability of $\mathcal{L}(\bullet)$. A frame property $P$ is said to be definable in a language $\mathcal{L}$, if there exists a set $\Theta$ of formulas in $\mathcal{L}$ such that $\mathcal{F}\vDash\Theta$ iff $\mathcal{F}$ has $P$. When $\Theta=\{\phi\}$, we write simply $\phi$ rather than $\{\phi\}$.

To demonstrate the undefinability of a frame property $P$ in a language $\mathcal{L}$, we (only) need to construct two frames such that one of them has $P$ but the other fails, and any $\mathcal{L}$-formula is valid on one frame if and only if it is also valid on the other. The argument is as follows: were $P$ defined by a set of $\mathcal{L}$-formulas $\Theta$, we would derive that $\mathcal{F}\vDash\Theta$ iff $\mathcal{F}$ has $P$. As any $\mathcal{L}$-formula is valid on one frame if and only if it is also valid on the other, this also applies to the set $\Theta$. This implies that one frame has $P$ iff the other also has, which is a contradiction.

\begin{proposition}\label{prop.undefinable-cr}
The frame properties $(c)$ and $(r)$ are undefinable in $\mathcal{L}(\bullet)$.
\end{proposition}

\begin{proof}
Consider the following frames:
$$
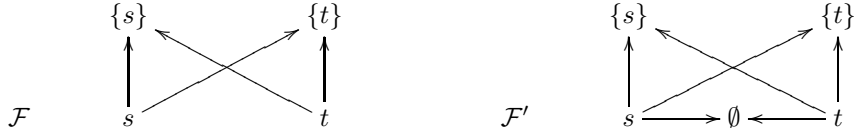
\xymatrix{&\{s\}&&\{t\}&&&\{s\}&&\{t\}\\
\mathcal{F}&s\ar[u]\ar[urr]&&t\ar[u]\ar[ull]&&\mathcal{F}'&s\ar[u]\ar[urr]\ar[r]&\emptyset&t\ar[u]\ar[ull]\ar[l]\\}
$$

From the above figure, we can see that $\mathcal{F}'$ possesses $(c)$ and $(r)$ but $\mathcal{F}$ does not, since $\{s\}\in N(s)$ and $\{t\}\in N(s)$ but $\{s\}\cap \{t\}=\emptyset\notin N(s)$.

Moreover, one may easily verify that $\mathcal{F}'=\mathcal{F}^{t^+}$ in which $\Gamma_s=\Gamma_t=\{\emptyset\}$, then by Coro.~\ref{coro.frameequivalent-circ}, we obtain that for all $\phi\in\mathcal{L}(\bullet)$, we have $\mathcal{F}\vDash\phi$ iff $\mathcal{F}'\vDash\phi$.
\end{proof}

\weg{\begin{proposition}
The frame properties $(m)$ and $(r)$ are undefinable in $\mathcal{L}(\bullet)$.
\end{proposition}

\begin{proof}
Consider the following frames:
$$
\xymatrix{&\{s\}&&&&&\{s\}&&\emptyset\\
\mathcal{F}&s\ar[u]\ar[r]&\{s,t\}&t\ar[l]&&\mathcal{F}'&s\ar[u]\ar[urr]\ar[r]&\{s,t\}&t\ar[u]\ar[ull]\ar[l]\\}
$$

At first, $\mathcal{F}$ has $(r)$ but $\mathcal{F}'$ does not, since $\{s\}\in N'(s)$ and $\{t\}\in N'(s)$ but $\emptyset\notin N'(s)$; secondly, $\mathcal{F}'$ has $(m)$ whereas $\mathcal{F}$ does not, because $\{s\}\in N(s)$ but $\{s,t\}\notin N(s)$ although $\{s\}\subseteq \{s,t\}$.

Besides, $\mathcal{F}'=\mathcal{F}^{t^+}$ where $\Gamma_s=\{\{t\}\}$ and $\Gamma_t=\{\{s\}\}$. Then by Prop.~\ref{coro.frameequivalent-circ}, we derive that $\mathcal{F}\vDash\phi$ iff $\mathcal{F}'\vDash\phi$ for all $\phi\in\mathcal{L}(\circ)$.
\end{proof}}

\begin{proposition}\label{prop.undefinable-m}
The frame property $(m)$ is undefinable in $\mathcal{L}(\bullet)$.
\end{proposition}

\begin{proof}
Consider the following frames:
$$
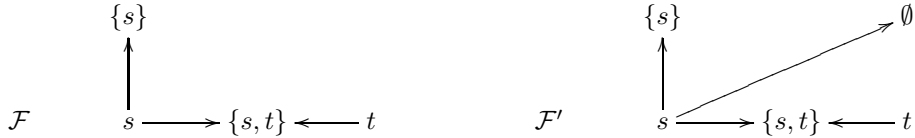
\xymatrix{&\{s\}&&&&&\{s\}&&\emptyset\\
\mathcal{F}&s\ar[u]\ar[r]&\{s,t\}&t\ar[l]&&\mathcal{F}'&s\ar[u]\ar[urr]\ar[r]&\{s,t\}&t\ar[l]\\}
$$

One may check that $\mathcal{F}$ possesses $(m)$ whereas $\mathcal{F}'$ does not, since $\emptyset\in N'(s)$ but $\{t\}\notin N'(s)$ although $\emptyset\subseteq\{t\}$.


Besides, $\mathcal{F}'=\mathcal{F}^{t^+}$ where $\Gamma_s=\{\emptyset\}$ and $\Gamma_t=\emptyset$. Then by Coro.~\ref{coro.frameequivalent-circ}, we derive that $\mathcal{F}\vDash\phi$ iff $\mathcal{F}'\vDash\phi$ for all $\phi\in\mathcal{L}(\bullet)$.
\end{proof}

Although the properties of $(m)$, $(c)$ and $(r)$ are undefinable in $\mathcal{L}(\bullet)$, the property $(n)$ is definable in the language. This can be explained via Coro.~\ref{coro.frameequivalent-circ} as follows: since for all $w$ in $\M=\lr{S,N,V}$, $w$ must be in $S$, thus it must be the case that $S\notin \Gamma_w$, and this makes a suitable definition of $\Gamma_w$ in showing the undefinability as in Props.~\ref{prop.undefinable-cr} and~\ref{prop.undefinable-m} unavailable.

\begin{proposition}\label{prop.definable-n}
The frame property $(n)$ is definable in $\mathcal{L}(\bullet)$.
\end{proposition}

\begin{proof}
We show that $(n)$ is defined by $\circ\top.$ Let $\mathcal{F}=\lr{S,N}$. 

Suppose that $\mathcal{F}$ has $(n)$, to show that $\mathcal{F}\vDash\circ\top$. For this, for any model $\M$ based on $\mathcal{F}$ and $s\in S$, we need to show that $\M,s\vDash\circ\top$, which amounts to showing that $S\in N(s)$ (because $\M,s\vDash\top$ and $\top^\M=S$). And $S\in N(s)$ is immediate by supposition.

Conversely, assume that $\mathcal{F}$ does not have $(n)$, then there exists $s\in S$ such that $S\notin N(s)$, that is, $\top^\M\notin N(s)$. We have also $\M,s\vDash\top$, and thus $\M,s\nvDash\circ\top$, therefore $\mathcal{F}\nvDash\circ\top$.
\end{proof}

\subsection{$W$-morphisms}

\weg{\begin{proposition}
On the class of all models, the $(m)$-models, the $(c)$-models, the $(n)$-models, the $(r)$-models, $\mathcal{L}(\circ)$ is not at least as expressive as $\mathcal{L}(W)$.
\end{proposition}

\begin{proof}
Consider the following models:
$$
\xymatrix{&&&&&&&&\{t\}\\
\M&s:\neg p\ar[r]&\{s,t\}&t:p\ar[l]&&\M'&s:\neg p\ar[r]\ar[urr]&\{s,t\}&t:p\ar[l]}
$$

Firstly, it may be easily checked that both $\M$ and $\M'$ have $(m)$, $(c)$, $(n)$ and $(r)$.

Secondly, $(\M,s)$ and $(\M',s)$ can be distinguished by an $\mathcal{L}(W)$-formula: on the one hand, as $p^\M=\{t\}\notin N(s)$, we have $\M,s\nvDash Wp$; on the other hand, since $\M',s\nvDash p$ and $p^{\M'}=\{t\}\in N'(s)$, we infer that $\M,s\vDash Wp$.

However, these two pointed models cannot be distinguished by any $\mathcal{L}(\circ)$-formulas. To see this, note that $\M'=\M^{t^+}$ where $\Gamma_s=\{\{t\}\}$ and $\Gamma_t=\{\emptyset\}$. Then using Prop.~\ref{prop.equivalence-circ}, we conclude that $\M,s\vDash\phi$ iff $\M',s\vDash\phi$ for all $\phi\in\mathcal{L}(\circ)$.
\end{proof}}


\begin{definition}[$W$-morphisms]
Let $\M=\lr{S,N,V}$ and $\M'=\lr{S',N',V'}$ be neighborhood models. A function $f:S\to S'$ is a {\em $W$-morphism} from $\M$ to $\M'$, if for all $s\in S$,
\begin{enumerate}
\item[(Var)] $s\in V(p)$ iff $s'\in V'(p)$ for all $p\in \BP$;
\item[($W$-Mor)] for all $X\subseteq S$, $[X\in N(s)\text{ and }s\notin X]\iff [f[X]\in N'(f(s))\text{ and }f(s)\notin f[X]].$
\end{enumerate}

We say that $\M'$ is a $W$-morphic image of $\M$, if there is a surjective $W$-morphism from $\M$ to $\M'$.
\end{definition}

\begin{proposition}\label{prop.w-morphisms}
Let $\M=\lr{S,N,V}$ and $\M'=\lr{S',N',V'}$ be neighborhood models, and let $f$ be a $W$-morphism from $\M$ to $\M'$. Then for all $s\in S$, for all $\phi\in\mathcal{L}(W)$, we have
$\M,s\vDash\phi\iff \M',f(s)\vDash\phi$, that is, $f[\phi^\M]=\phi^{\M'}$.
\end{proposition}

\begin{proof}
By induction on $\phi$, where the only nontrivial case is $W\phi$.

Suppose that $\M,s\vDash W\phi$, to show that $\M',f(s)\vDash W\phi$. By supposition, $\phi^\M\in N(s)$ and $s\notin \phi^\M$. By ($W$-Mor), $f[\phi^\M]\in N'(f(s))$ and $f(s)\notin f[\phi^\M]$. By induction hypothesis, we infer that $\phi^{\M'}\in N'(f(s))$ and $f(s)\notin \phi^{\M'}$, and thus $\M',f(s)\vDash W\phi$.

Conversely, assume that $\M',f(s)\vDash W\phi$, to prove that $\M,s\vDash W\phi$. By assumption, $\phi^{\M'}\in N'(f(s))$ and $f(s)\notin \phi^{\M'}$. By induction hypothesis, we derive that $f[\phi^\M]\in N'(f(s))$ and $f(s)\notin f[\phi^\M]$. Then by ($W$-Mor) again, we get $\phi^\M\in N(s)$ and $s\notin \phi^\M$, and therefore $\M,s\vDash W\phi$.
\end{proof}

The models $\M^{u^+}$ and $\M^{u^-}$ defined in~\cite[p.~262]{GilbertVenturi:2017} are, respectively, the special cases of those defined in the following proposition, when $\Sigma_w=U_w$. Therefore, the following proposition extends the result in~\cite[Thm.~2.8]{GilbertVenturi:2017}.

\begin{proposition}\label{prop.equivalence-W}
Let $\M=\lr{S,N,V}$. For all $w\in S$ and $\alpha\in\mathcal{L}(W)$, we have
$$\M,w\vDash \alpha\text{ iff }\M^{u^+},w\vDash\alpha$$
and
$$\M,w\vDash \alpha\text{ iff }\M^{u^-},w\vDash\alpha,$$
where $\M^{u^+}=\lr{S,N^{u^+},V}$ and $\M^{u^-}=\lr{S,N^{u^-},V}$, where $N^{u^+}(w)=N(w)\cup \Sigma_w$ and $N^{u^-}(w)=N(w)\backslash \Sigma_w$ for $\Sigma_w\subseteq U_w=\{X\subseteq S\mid w\in X\}$.
\end{proposition}

\begin{proof}
By Prop.~\ref{prop.w-morphisms}, it suffices to show that $f:S\to S$ such that $f(x)=x$ is a $W$-morphism from $\M$ to $\M^{u^+}$, and also a $W$-morphism from $\M$ to $\M^{u^-}$.

The condition (Var) is clear. For ($W$-Mor), we only need to show that
$$[X\in N(w)\text{ and }w\notin X]\iff [X\in N^{u^+}(w)\text{ and }w\notin X]~~~~~~~~~~~~~~~~~(1)$$
and
$$[X\in N(w)\text{ and }w\notin X]\iff [X\in N^{u^-}(w)\text{ and }w\notin X]~~~~~~~~~~~~~~~~~(2).$$

The ``$\Longrightarrow$'' of (1) and ``$\Longleftarrow$'' of (2) are straightforward since $N^{u^-}(w)\subseteq N(w)\subseteq N^{u^+}(w)$.

Moreover, if $w\notin X$, then $X\notin U_w$, thus $X\notin \Sigma_w$. This gives us  ``$\Longleftarrow$'' of (1) and ``$\Longrightarrow$'' of (2).
\end{proof}

\weg{\begin{proof}
By Prop.~\ref{prop.inv-bis-W}, it suffices to show that $(\M,w)\bis_W(\M^{u^+},w)$ and $(\M,w)\bis_W(\M^{u^-},w)$. For this, define $Z=\{(w,w)\mid w\in S\}$. Let $wZw$. The condition (Var) is clear.

For (W-Coh), suppose that $(X,X')$ is $Z$-coherent. Then it must be the case that $X=X'$: otherwise, there would be an $w\in S$ such that $w\in X$ but $w\notin X'$, which contradicts the fact that $wZw$ and the supposition. We only need to show that
$$[X\in N(w)\text{ and }w\notin X]\iff [X\in N^{u^+}(w)\text{ and }w\notin X]~~~~~~~~~~~~~~~~~(1)$$
and
$$[X\in N(w)\text{ and }w\notin X]\iff [X\in N^{u^-}(w)\text{ and }w\notin X]~~~~~~~~~~~~~~~~~(2).$$

The ``$\Longrightarrow$'' of (1) and ``$\Longleftarrow$'' of (2) are straightforward since $N^{u^-}(w)\subseteq N(w)\subseteq N^{u^+}(w)$.

Moreover, if $w\notin X$, then $X\notin U_w$, thus $X\notin \Sigma_w$. This gives us  ``$\Longleftarrow$'' of (1) and ``$\Longrightarrow$'' of (2).
\end{proof}}

Similar to the case in Prop.~\ref{prop.equivalence-circ}, here $\Sigma_w$ is defined in terms of $w$, thus given any two points $x,y\in S$, $\Sigma_x$ may be different from $\Sigma_y$.

Now coming back to Prop.~\ref{prop.exp-lw-lcirc}, without showing directly $\mathcal{L}(W)$-formulas cannot distinguish between $(\M,s)$ and $(\M',s')$, we can appeal to Prop.~\ref{prop.equivalence-W}, by noting that $\M'=\M^{u^+}$ where $\Sigma_s=\{\{s\}\}$ and $\Sigma_t=\{\{s,t\}\}$.\footnote{Note that since $\Sigma_s$ is an arbitrary subset of $U_s$, and $\{s\}\in U_s$ (as $s\in \{s\}$), then we can set $\Sigma_s=\{\{s\}\}$. Similar arguments also holds for $\Sigma_t$ and other definitions of $\Sigma_s$ and $\Sigma_t$ in other situations below.} Prop.~\ref{prop.equivalence-W} will be also useful in proving a general completeness result (see Thm.~\ref{thm.generalcomp}).

With Prop.~\ref{prop.equivalence-W}, we have immediately the following, which extends the result in~\cite[Coro.~2.9]{GilbertVenturi:2017}.
\begin{corollary}\label{coro.frameequivalent-W}
Let $\mathcal{F}=\lr{S,N}$, and $\mathcal{F}^{u^+}=\lr{S,N^{u^+}}$ and $\mathcal{F}^{u^-}=\lr{S,N^{u^-}}$ be defined as in Prop.~\ref{prop.equivalence-W}. Then for all $\phi\in\mathcal{L}(\circ)$, we have
$$\mathcal{F}\vDash\phi\text{ iff }\mathcal{F}^{u^+}\vDash\phi$$
and
$$\mathcal{F}\vDash\phi\text{ iff }\mathcal{F}^{u^-}\vDash\phi.$$
\end{corollary}

Similar to Coro.~\ref{coro.frameequivalent-circ}, Coro.~\ref{coro.frameequivalent-W} can also be applied to proving the results of frame (un)definability in $\mathcal{L}(W)$.
\weg{\begin{proposition}
The frame properties $(m)$, $(n)$, $(f)$ are undefinable in $\mathcal{L}(W)$.
\end{proposition}

\begin{proof}
Consider the following frames:
\[
\xymatrix{\mathcal{F}&s\ar[r]&\emptyset&&\mathcal{F}'&s\ar[r]&\{s\}\\}
\]

One may check that $\mathcal{F}$ has $(f)$, but $\mathcal{F}'$ fails, since $\{s\}\in N'(s)$ but $s\in \{s\}$; $\mathcal{F}'$ has $(m)$ and $(n)$, but $\mathcal{F}$ does not, since $\emptyset\in N(s)$ but $\{s\}\notin N(s)$.

Moreover, $\mathcal{F}'=\mathcal{F}^{u^+}$ where $\Sigma_s=\{\{s\}\}$. Then by Coro.~\ref{coro.frameequivalent-W}, we conclude that for all $\phi\in\mathcal{L}(W)$, $\mathcal{F}\vDash\phi\iff \mathcal{F}'\vDash\phi.$
\end{proof}}

\begin{proposition}
The frame properties $(m)$ and $(n)$ are undefinable in $\mathcal{L}(W)$.
\end{proposition}

\begin{proof}
Consider the following frames:
\[
\xymatrix{\mathcal{F}&s\ar[r]&\emptyset&&\mathcal{F}'&\{s\}&s\ar[r]\ar[l]&\emptyset\\}
\]

One may check that $\mathcal{F}'$ has $(m)$ and $(n)$, but $\mathcal{F}$ does not, since $\emptyset\in N(s)$ but $\{s\}\notin N(s)$ although $\emptyset\subseteq \{s\}$.

Moreover, $\mathcal{F}'=\mathcal{F}^{u^+}$ where $\Sigma_s=\{\{s\}\}$. Then by Coro.~\ref{coro.frameequivalent-W}, we conclude that for all $\phi\in\mathcal{L}(W)$, $\mathcal{F}\vDash\phi\iff \mathcal{F}'\vDash\phi.$
\end{proof}

\begin{proposition}
The frame properties $(c)$ and $(r)$ are undefinable in $\mathcal{L}(W)$.
\end{proposition}

\begin{proof}
Consider the following frames:
$$
\xymatrix{&&&\{t\}&&&\{s\}&&\{t\}\\
\mathcal{F}&s\ar[r]\ar[urr]&\{s,t\}&t\ar[l]\ar[u]&&\mathcal{F}'&s\ar[r]\ar[u]\ar[urr]&\{s,t\}&t\ar[l]\ar[u]}
$$

One may check that $\mathcal{F}$ has $(c)$ and $(r)$, but $\mathcal{F}'$ fails, since $\{s\}\in N'(s)$ and $\{t\}\in N'(s)$ but $\{s\}\cap\{t\}=\emptyset\notin N'(s)$.

Moreover, $\mathcal{F}'=\mathcal{F}^{u^+}$ where $\Sigma_s=\{\{s\}\}$ and $\Sigma_t=\{\{s,t\}\}$. Then by Coro.~\ref{coro.frameequivalent-W}, we conclude that for all $\phi\in\mathcal{L}(W)$, $\mathcal{F}\vDash\phi\iff \mathcal{F}'\vDash\phi.$
\end{proof}

We conclude this part with another application of the notion of $W$-morphisms. For this, we define the notion of transitive closure of a neighborhood frame, which comes from~\cite[Def.~2.12]{GilbertVenturi:2017}.

\begin{definition}
Given a neighborhood frame $\mathcal{F}=\lr{S,N}$, we define its {\em transitive closure} $\mathcal{F}^{tc}=\lr{S,N^{tc}}$ inductively as $\bigcup_{i\in\mathbb{N}}\mathcal{F}_i$, with $\mathcal{F}_0=\mathcal{F}$ and $\mathcal{F}_{i+1}=\lr{S,N_{i+1}}$, where 
$$N_{i+1}(w)=N_i(w)\cup\{m_{N_i}(X)\mid X\in N_i(w)\}$$
for every $w\in S$, and
$$m_{N_i}(X)=\{z\in S\mid X\in N_i(z)\}$$
for $X\subseteq S$.
\end{definition}

\begin{fact}\label{fact.literature}\cite[Fact~2.13]{GilbertVenturi:2017}
For all $w\in S$, if $X\in N^{tc}(w)\backslash N(w)$, then $w\in X$.
\end{fact}

The following proposition is shown in~\cite[Thm.~2.14]{GilbertVenturi:2017}, but without use of a morphism argument. Here we give a much easier proof via the notion of $W$-morphisms.
\begin{proposition}
Let $\M=\lr{S,N,V}$ be a model based on a frame $\mathcal{F}$ and $\M^{tc}$ the corresponding one based on $\mathcal{F}^{tc}$. For all $w\in S$ and $\phi\in\mathcal{L}(W)$, we have
$$\M,w\vDash\phi\text{ iff }\M^{tc},w\vDash\phi.$$
\end{proposition}

\begin{proof} We show a stronger result: $f:S\to S$ such that $f(x)=x$ is a $W$-morphism from $\M$ to $\M^{tc}$.
which implies the statement due to Prop.~\ref{prop.w-morphisms}.  The condition (Var) is straightforward.

For ($W$-Mor), we need to show that
$$[X\in N(x)\text{ and }x\notin X]\iff [X\in N^{tc}(x)\text{ and }x\notin X].$$

The `$\Longrightarrow$' follows immediately since $N(x)\subseteq N^{tc}(x)$. For the other direction, if $X\in N^{tc}(x)\text{ and }x\notin X$, by Fact~\ref{fact.literature}, we obtain that $X\in N(x)$, as desired.
\end{proof}

\section{Axiomatizations}\label{sec.axiomatizations}

We now axiomatize $\mathcal{L}(\bullet)$ and $\mathcal{L}(W)$ over various neighborhood frames.

\subsection{Axiomatizations for $\mathcal{L}(\bullet)$}

The following lists the axioms and inference rules that are needed in this part.
\[\begin{array}{lllll}
\text{Axioms}&&&\text{Rules}\\
\texttt{PL}& \text{All instances of propositional tautologies}&&&\\
\circ\texttt{E}& \bullet\phi\to\phi&& \texttt{MP}&\dfrac{\phi,\phi\to\psi}{\psi}\\
\circ\texttt{M}& \circ\phi\land\phi\to\circ(\phi\vee\psi) && \\
\circ\texttt{C}& \circ\phi\land\circ\psi\to\circ(\phi\land\psi)&&\texttt{RE}\circ&\dfrac{\phi\lra \psi}{\circ\phi\lra \circ\psi}\\
\circ\texttt{N}& \circ\top &&\\
\end{array}\]

All axioms and inference rules arise in the literature, with distinct names, except for $\circ\texttt{M}$, which is derivable from axiom (K1.2) in~\cite{Marcos:2005}, that is, $((\phi\land\circ\phi)\vee(\psi\land\circ\psi))\to \circ(\phi\vee\psi)$. Rather, a stronger rule $\dfrac{\phi\to\psi}{(\circ\phi\land\phi)\to(\circ\psi\land\psi)}$ (denoted $\texttt{RM}\circ$), has usually been used to replace axiom $\circ\texttt{M}$ (see e.g.~\cite{Steinsvold:2008a,GilbertVenturi:2016,GilbertVenturi:2017}). But we prefer axioms to rules of inference. As we will see below, given $\texttt{RE}\circ$ (and propositional calculus), the rule $\texttt{RM}\circ$ is also derivable from $\circ\texttt{M}$.

\begin{proposition}\label{prop.derivable-rm}
$\texttt{RM}\circ$ is derivable from $\texttt{PL}+\texttt{MP}+\circ\texttt{M}+\texttt{RE}\circ$.
\end{proposition}

\begin{proof}
We have the following proof sequences in $\texttt{PL}+\texttt{MP}+\circ\texttt{M}+\texttt{RE}\circ$:
\[
\begin{array}{llll}
(1)& \phi\to \psi & &\text{Premise}\\
(2)& \phi\vee\psi\lra \psi && (1),\texttt{PL},\texttt{MP}\\
(3)& \circ(\phi\vee\psi)\lra \circ\psi & & (2),\texttt{RE}\circ\\
(4)& \circ\phi\land\phi\to\circ(\phi\vee\psi) & & \circ\texttt{M}\\
(5)& \circ\phi\land\phi\to\circ\psi & & (3),(4)\\
(6)& (\circ\phi\land\phi)\to(\circ\psi\land\psi) & & (1),(5)\\
\end{array}
\]
\end{proof}


If we consider all axioms and rules above, we obtain a logic called ${\bf B_K}$ in~\cite{Steinsvold:2008a,GilbertVenturi:2016,GilbertVenturi:2017}.\footnote{More precisely, the system ${\bf B_K}$ contains the rule $\texttt{RM}\circ$ instead of the axiom $\circ\texttt{M}$, and skips the rule $\texttt{RE}\circ$ since it is then derivable from $\texttt{RM}\circ$ and $\circ\texttt{E}$ (see~\cite[Prop.~3.2]{GilbertVenturi:2016}).} ${\bf B_K}$ is the minimal logic for $\mathcal{L}(\circ)$ over relational semantics, that is, it is sound and strongly complete with respect to the class of all relational frames~\cite{Steinsvold:2008a}. As each Kripke model is pointwise equivalent to some augmented model, ${\bf B_K}$ is also (sound and) strongly complete with respect to the class of augmented frames~\cite{GilbertVenturi:2017}. Moreover, since every augmented model is a filter, thus ${\bf B_K}$ also characterizes the class of filters. From now on, for the sake of consistency on notation, we denote the logic by ${\bf K^\circ}$ here. As neighborhood semantics can handle logics weaker than the minimal relational logic, it is then natural to ask what logics weaker than ${\bf K^\circ}$ look like. Here is a table that summarizes ${\bf K^\circ}$ and its weaker logics and the corresponding classes of frames which determine them.\footnote{It is worth remarking that $\circ\texttt{E}$ is indispensable in any proof system in the table. To see this, define a new semantics which interprets all formulas of the form $\circ\phi$ as $\phi$ (so that $\bullet\phi$ is interpreted as $\neg\phi$), then one can see that under the new semantics, $\circ\texttt{E}$ is not valid, but any subsystem ${\bf L}-\circ\texttt{E}$ of ${\bf L}$ in the table is sound. This entails that $\circ\texttt{E}$ is not derivable in any such subsystem, and thus $\circ\texttt{E}$ is indispensable in any proof system in the table.}
\[\begin{array}{|c|c|}
\hline
\text{Proof systems}&\text{Frame classes}\\
\hline
  {\bf E^\circ}=\texttt{PL}+\texttt{MP}+\circ\texttt{E}+\texttt{RE}\circ& \text{---}\\
  {\bf M^\circ}={\bf E^\circ}+\circ\texttt{M}& (m)\\
  {\bf EC^\circ}={\bf E^\circ}+\circ\texttt{C}& (c)\\
  {\bf EN^\circ}={\bf E^\circ}+\circ\texttt{N} & (n)\\
  {\bf EMC^\circ}={\bf M^\circ}+\circ\texttt{C}& (mc)\\
  {\bf EMN^\circ}={\bf M^\circ}+\circ\texttt{N}& (mn)\\
  {\bf ECN^\circ}={\bf EC^\circ}+\circ\texttt{N}& (cn)\\
  {\bf K^\circ}={\bf EMC^\circ}+\circ\texttt{N}& \text{filters}=(mcn)\\
\hline
\end{array}
\]

A natural question is: are all unknown truths themselves unknown truths? Interestingly, in monotone logics, the answer is positive. We now give a proof-theoretical perspective.
\begin{proposition}\label{prop.unknowntruths}
$\bullet\phi\to\bullet\bullet\phi$ is provable in ${\bf M^\circ}$.
\end{proposition}

\begin{proof}
Notice that we have the following proof sequences in ${\bf M^\circ}$.
\[
\begin{array}{lll}
(i) & \bullet\phi\to\phi&\circ\texttt{E}\\
(ii)& \circ\bullet\phi\land\bullet\phi\to \circ\phi&(i),\texttt{RM}\circ\\
(iii)&\circ\bullet\phi\to\circ\phi&(ii),\texttt{PL}\\
(iv) &\bullet\phi\to\bullet\bullet\phi&(iii),\texttt{PL}
\end{array}
\]
\end{proof}


We now focus on the completeness of the proof systems in the above table. The completeness proof is based on the construction of the canonical model. From now on, we define the {\em proof set} of $\phi$ in a system $\Lambda$, denoted $|\phi|_\Lambda$, as the set of maximal consistent sets of $\Lambda$ that contains $\phi$; in symbol, $|\phi|_\Lambda=\{s\in S^c\mid \phi\in s\}$. We skip the subscript and simply write $|\phi|$ whenever the system $\Lambda$ is clear. If a set of states $\Gamma$ is not a proof set in $\Lambda$ for any formula, then we say it is a {\em non-proof set} relative to $\Lambda$.
\begin{definition}\label{def.cm-ecirc}
The {\em canonical model} for ${\bf E^\circ}$ is the triple $\M^c=\lr{S^c,N^c,V^c}$, where
\begin{itemize}
\item $S^c=\{s\mid s\text{ is a maximal consistent set of }{\bf E^\circ}\}$,
\item $N^c(s)=\{|\phi|\mid \circ\phi\land\phi\in s\}$,
\item $V^c(p)=\{s\in S^c\mid p\in s\}$.
\end{itemize}
\end{definition}


\weg{\begin{lemma}
For all $s\in S^c$, for all $\phi\in\mathcal{L}(\circ)$, we have
$\M^c,s\vDash\phi\iff \phi\in s$, that is, $\phi^{\M^c}=|\phi|$.
\end{lemma}

\begin{proof}
By induction on $\phi$. The nontrivial case is $\circ\phi$, that is, we only need to show that $\M^c,s\vDash \circ\phi$ iff $\circ\phi\in s$.

First, suppose that $\circ\phi\in s$, to show that $\M^c,s\vDash \circ\phi$. Assume that $\M^c,s\vDash\phi$, viz., $s\in\phi^{\M^c}$, then by induction hypothesis, $s\in |\phi|$, namely, $\phi\in s$. By supposition, we infer that $\circ\phi\land\phi\in s$. Then from the definition of $N^c$, it follows that $|\phi|\in N^c(s)$. Now by induction hypothesis again, we conclude that $\phi^{\M^c}\in N^c(s)$. Therefore, $\M^c,s\vDash \circ\phi$.

Conversely, suppose that $\circ\phi\notin s$, to prove that $\M^c,s\nvDash \circ\phi$, which by induction hypothesis is equivalent to showing that $\phi\in s$ and $|\phi|\notin N^c(s)$. By supposition, we have $\bullet\phi\in s$. Then using axiom $\circ\texttt{E}$, we infer that $\phi\in s$. As $\circ\phi\land\phi\notin s$, we derive that $|\phi|\notin N^c(s)$ according to the definition of $N^c$.
\end{proof}}

\begin{lemma}
For all $s\in S^c$, for all $\phi\in\mathcal{L}(\bullet)$, we have
$\M^c,s\vDash\phi\iff \phi\in s$, that is, $\phi^{\M^c}=|\phi|$.
\end{lemma}

\begin{proof}
By induction on $\phi$. The nontrivial case is $\bullet\phi$, that is, we only need to show that $\M^c,s\vDash \bullet\phi$ iff $\bullet\phi\in s$.

First, suppose that $\bullet\phi\in s$, to prove that $\M^c,s\vDash \bullet\phi$, which by induction hypothesis is equivalent to showing that $\phi\in s$ and $|\phi|\notin N^c(s)$. By supposition and axiom $\circ\texttt{E}$, we infer that $\phi\in s$. As $\bullet\phi\in s$, we have $\circ\phi\land\phi\notin s$, and then $|\phi|\notin N^c(s)$ according to the definition of $N^c$.

Conversely,  suppose that $\bullet\phi\notin s$, to show that $\M^c,s\nvDash \bullet\phi$. Assume that $\M^c,s\vDash\phi$, viz., $s\in\phi^{\M^c}$, then by induction hypothesis, $s\in |\phi|$, namely, $\phi\in s$. By supposition, we infer that $\circ\phi\land\phi\in s$. Then from the definition of $N^c$, it follows that $|\phi|\in N^c(s)$. Now by induction hypothesis again, we conclude that $\phi^{\M^c}\in N^c(s)$. Therefore, $\M^c,s\nvDash \bullet\phi$.
\end{proof}

We also need to show that $N^c$ is well-defined.
\begin{lemma}
If $|\phi|\in N^c(s)$ and $|\phi|=|\psi|$, then $\circ\psi\land\psi\in s$.
\end{lemma}

\begin{proof}
Suppose that $|\phi|\in N^c(s)$ and $|\phi|=|\psi|$, to show that $\circ\psi\land\psi\in s$. By supposition, we obtain $\circ\phi\land\phi\in s$ and $\vdash\phi\lra\psi$. By $\texttt{RE}\circ$, it follows that $\vdash\circ\phi\lra\circ\psi$. Therefore, $\circ\psi\land\psi\in s$.
\end{proof}

Now it is a routine work to show the following.
\begin{theorem}\label{thm.e-comp}
${\bf E^\circ}$ is sound and strongly complete with respect to the class of all neighborhood frames.
\end{theorem}


\begin{theorem}\label{thm.ec-comp}
${\bf EC^\circ}$ is sound and strongly complete with respect to the class of $(c)$-frames.
\end{theorem}

\begin{proof}
For soundness, we need to show the validity of $\circ\texttt{C}$ over the class of $(c)$-frames. For this, let $\M=\lr{S,N,V}$ be a $(c)$-model, $s\in S$, and suppose that $\M,s\vDash\circ\phi\land\circ\psi$, to show that $\M,s\vDash\circ(\phi\land\psi)$. Assume that $s\in (\phi\land\psi)^\M$, it suffices to show that $(\phi\land\psi)^\M\in N(s)$. By supposition, it follows that $s\in \phi^\M$ implies $\phi^\M\in N(s)$, and $s\in \psi^\M$ implies $\psi^\M\in N(s)$. By assumption, $s\in \phi^\M$ and $s\in \psi^\M$, and thus $\phi^\M\in N(s)$ and $\psi^\M\in N(s)$. An application of $(c)$ gives us $\phi^\M\cap \psi^\M\in N(s)$, that is, $(\phi\land\psi)^\M\in N(s)$, as desired.

\medskip

For completeness, define $\M^c$ w.r.t. ${\bf EC^\circ}$ as in Def.~\ref{def.cm-ecirc}. It suffices to show that $N^c$ is closed under conjunctions. For this, let $s\in S^c$ be arbitrary, and suppose that $X\in N^c(s)$ and $Y\in N^c(s)$, to show that $X\cap Y\in N^c(s)$. By supposition, there are $\phi,\psi$ such that $X=|\phi|\in N^c(s)$ and $Y=|\psi|\in N^c(s)$, then $\circ\phi\land\phi\in s$ and $\circ\psi\land\psi\in s$. From this and axiom $\circ\texttt{C}$, it follows that $\circ(\phi\land\psi)\land (\phi\land\psi)\in s$, and thus $|\phi\land\psi|\in N^c(s)$, viz. $X\cap Y\in N^c(s)$.
\end{proof}



\begin{theorem}\label{thm.en-comp}
${\bf EN^\circ}$ is sound and strongly complete with respect to the class of $(n)$-frames.
\end{theorem}

\begin{proof}
The soundness follows directly from the soundness of ${\bf E^\circ}$ (Thm.~\ref{thm.e-comp}) and the validity of $\circ\texttt{N}$ (Prop.~\ref{prop.definable-n}).

For the completeness, define $\M^c$ w.r.t. ${\bf EN^\circ}$ as in Def.~\ref{def.cm-ecirc}. It suffices to show that for all $s\in S^c$, $S^c\in N^c(s)$. This follows immediately from the axiom $\circ\top\land\top\in s$ and the fact that $|\top|=S^c$.
\end{proof}

The following is a consequence of Thm.~\ref{thm.ec-comp} and Thm.~\ref{thm.en-comp}.
\begin{corollary}\label{coro.ecn-comp}
${\bf ECN^\circ}$ is sound and strongly complete with respect to the class of $(cn)$-frames.
\end{corollary}

Now we deal with the completeness of ${\bf M^\circ}$. As in the case of monotone modal logic, the canonical neighborhood function $N^c$ is not necessarily supplemented due to the presence of non-proof sets. To deal with this problem, we use the strategy of supplementation.

\begin{definition}\label{def.supplementation}
Let $\M=\lr{S,N,V}$ be a neighborhood model. We say that $\M^+=\lr{S,N^+,V}$ is the {\em supplementation} of $\M$, if for all $s\in S$, $N^+(s)=\{X\mid Y\subseteq X\text{ for some }Y\in N(s)\}$.
\end{definition}

Given any neighborhood model, its supplementation is supplemented. Also, $N(s)\subseteq N^+(s)$ for all $s\in S$. Moreover, the supplementation preserves the properties $(c)$ and $(n)$.
\begin{fact}\label{fact.MtoM+}
Let $\M=\lr{S,N,V}$ be a neighborhood frame. If $\M$ has $(c)$, then so does $\M^+$; if $\M$ has $(n)$, then so does $\M^+$.
\end{fact}

\begin{proof}
Suppose that $\M$ has $(c)$. Let $s\in S$ and $X,X'\subseteq S$, if $X,X'\in N^+(s)$, then $Y\subseteq X$ and $Y'\subseteq X'$ for some $Y,Y'\in N(s)$, thus $Y\cap Y'\subseteq X\cap X'$. From $Y,Y'\in N(s)$ and the supposition, it follows that $Y\cap Y'\in N(s)$. Therefore, $X\cap X'\in N^+(s)$. This means that $\M^+$ has also $(c)$.

Assume that $\M$ has $(n)$. Then $S\in N(s)$ for all $s\in S$. Since $N(s)\subseteq N^+(s)$, thus $S\in N^+(s)$ for all $s\in S$. This entails that $\M^+$ has also $(n)$.
\end{proof}


\begin{theorem}\label{thm.em-comp}
${\bf M^\circ}$ is sound and strongly complete with respect to the class of $(m)$-frames.
\end{theorem}

\begin{proof}
For soundness, by the soundness of ${\bf E^\circ}$ (Thm.~\ref{thm.e-comp}), it suffices to show that the axiom $\circ\texttt{M}$ preserves validity over $(m)$-frames.


Suppose for any $(m)$-model $\M=\lr{S,N,V}$ and $s\in S$ that $\M,s\vDash\circ\phi\land\phi$, to prove that $\M,s\vDash\circ(\phi\vee\psi)$. By supposition, we obtain that $s\in \phi^\M$ and $\phi^\M\in N(s)$. Since $\phi^\M\subseteq \phi^\M\cup \psi^\M$, by $(m)$, it follows that $(\phi\vee\psi)^\M\in N(s)$, and therefore $\M,s\vDash\circ(\phi\vee\psi)$.

\medskip

For completeness, define $\M^c$ w.r.t. ${\bf M^\circ}$ as in Def.~\ref{def.cm-ecirc}, and consider the supplementation of $\M^c$, that is, $(\M^c)^+=\lr{S^c,(N^c)^+,V^c}$. By definition of supplementation,  $(\M^c)^+$ possesses $(m)$.

It suffices to show that for all $s\in S^c$, for all $\phi\in\mathcal{L}(\circ)$,
$$|\phi|\in (N^c)^+(s)\iff \circ\phi\land\phi\in s.$$

``$\Longleftarrow$'' follows directly from the fact that $N^c(s)\subseteq (N^c)^+(s)$.

For ``$\Longrightarrow$'', suppose that $|\phi|\in (N^c)^+(s)$, then $X\subseteq |\phi|$ for some $X\in N^c(s)$. Since $X\in N^c(s)$, there must be a $\chi$ such that $|\chi|=X\in N^c(s)$, and then $\circ\chi\land\chi\in s$. We have also $|\chi|\subseteq |\phi|$, then $\vdash \chi\to\phi$. Note that the rule $\texttt{RM}\circ$ is derivable in ${\bf M^\circ}$ (Prop.~\ref{prop.derivable-rm}), thus we have $\vdash\circ\chi\land\chi\to\circ\phi\land\phi$, thus $\circ\phi\land\phi\in s$, as desired.
\end{proof}

\begin{theorem}
${\bf EMC^\circ}$ is sound and strongly complete with respect to the class of $(mc)$-frames.
\end{theorem}

\begin{proof}
The soundness follows directly from the soundness of ${\bf M^\circ}$ (Thm.~\ref{thm.em-comp}) and the validity of $\circ\texttt{C}$ (Thm.~\ref{thm.ec-comp}).

As for the completeness, define $\M^c$ and $(\M^c)^+$ w.r.t. ${\bf EMC^\circ}$ as in Thm.~\ref{thm.em-comp}. By Thm.~\ref{thm.em-comp}, it suffices to show that $(\M^c)^+$ possesses $(c)$. This follows immediately from Thm.~\ref{thm.ec-comp} and Fact~\ref{fact.MtoM+}.
\end{proof}

\begin{theorem}\label{thm.emn-comp}
${\bf EMN^\circ}$ is sound and strongly complete with respect to the class of $(mn)$-frames.
\end{theorem}

\begin{proof}
The soundness follows directly from the soundness of ${\bf M^\circ}$ (Thm.~\ref{thm.em-comp}) and the validity of $\circ\texttt{N}$ (Prop.~\ref{prop.definable-n}).

As for the completeness, define $\M^c$ and $(\M^c)^+$ w.r.t. ${\bf EMN^\circ}$ as in Thm.~\ref{thm.em-comp}. By Thm.~\ref{thm.em-comp}, it suffices to show that $(\M^c)^+$ possesses $(n)$. This follows immediately from Thm.~\ref{thm.en-comp} and Fact~\ref{fact.MtoM+}.
\end{proof}

\begin{theorem}
${\bf K^\circ}$ is sound and strongly complete with respect to the class of filters.
\end{theorem}

\begin{proof}
The soundness follows immediately from that of ${\bf EMN^\circ}$ (Thm.~\ref{thm.emn-comp}) and the validity of $\circ\texttt{C}$ (Thm.~\ref{thm.ec-comp}).

As for the completeness, define $\M^c$ and $(\M^c)^+$ w.r.t. ${\bf K^\circ}$ as in Thm.~\ref{thm.emn-comp}. By Thm.~\ref{thm.emn-comp}, it suffices to show that $(\M^c)^+$ has $(c)$. This follows from Thm.~\ref{thm.ec-comp} and Fact~\ref{fact.MtoM+}.\footnote{Note that there was a mistake in~\cite[Thm.~1.8]{GilbertVenturi:2017}, where the authors did not prove that $(\M^c)^+$ (denoted $\M^+$ there) has $(c)$ and $(n)$; rather, they only show that $\M^c$ (denoted $\M$ there) does have, which though does not directly give us the completeness in question.}
\end{proof}

\subsection{Axiomatizations for $\mathcal{L}(W)$}\label{subsec.axiomatizationforlw}


To axiomatize $\mathcal{L}(W)$ over various neighborhood frames, we list the following axioms and rules of inference.
\[\begin{array}{lllll}
\text{Axioms}&&&\text{Rules}\\
\texttt{PL}& \text{All instances of propositional tautologies}&&&\\
\texttt{WE}& W\phi\to \neg\phi&&\texttt{MP}&\dfrac{\phi,\phi\to\psi}{\psi}\\
\texttt{WM}& W(\phi\land\psi)\land\neg\psi\to W\psi\\
\texttt{WC}& W\phi\land W\psi\to W(\phi\land\psi)&&\texttt{REW}&\dfrac{\phi\lra\psi}{W\phi\lra W\psi}\\
\end{array}\]

Similar to the axiomatizations for $\mathcal{L}(\bullet)$, all axioms and inference rules listed above also arise in the literature, with different names. The axiom $\texttt{WM}$ is derivable from a rule $\dfrac{\phi\to\psi}{(W\phi\land\neg\psi)\to W\psi}$ (see~\cite[Thm.~3.2]{Steinsvold:falsebelief}), denoted $\texttt{RMW}$, which has usually been used to replace $\texttt{WM}$~\cite{Steinsvold:falsebelief,GilbertVenturi:2017}. Again, we prefer axioms to inference rules. Also, note that the rule $\texttt{RMW}$ is derivable from the axiom $\texttt{WM}$ in the presence of $\texttt{REW}$ (and propositional calculus).

\begin{proposition}\label{prop.derivable-rmw}
$\texttt{RMW}$ is derivable from $\texttt{PL}+\texttt{MP}+\texttt{WM}+\texttt{REW}$.
\end{proposition}

\begin{proof}
We have the following proof sequences in $\texttt{PL}+\texttt{MP}+\texttt{WM}+\texttt{REW}$.
\[
\begin{array}{lll}
(1)& \phi\to\psi& \text{Premise}\\
(2)& \phi\land\psi\lra \phi &(1), \texttt{PL},\texttt{MP}\\
(3)& W(\phi\land\psi)\lra W\phi& (2),\texttt{REW}\\
(4)& W(\phi\land\psi)\land \neg\psi\to W\psi & \texttt{WM}\\
(5)& (W\phi\land\neg\psi)\to W\psi& (3),(4)\\
\end{array}
\]
\end{proof}

It is shown that the proof system consisting of all axioms and inference rules for $\mathcal{L}(W)$, denoted ${\bf K^W}$ here, is sound and strongly complete with respect to the class of all relational frames in~\cite{Steinsvold:falsebelief} and to the class of all neighborhood frames that are closed under intersections and are negatively supplemented in~\cite{GilbertVenturi:2017}.\footnote{More precisely, the system in~\cite{Steinsvold:falsebelief} (called $S^W$ there) and~\cite{GilbertVenturi:2017} (called ${\bf A_K}$ there) contains the rule $\texttt{RMW}$ instead of the axiom $\texttt{WM}$, and drops the rule $\texttt{REW}$ since it is then derivable from $\texttt{RMW}$ and $\texttt{WE}$ (see~\cite[Thm.~3.1]{Steinsvold:falsebelief}).} We will give the definition of `negatively supplemented' later. Again, it is natural to ask what logics weaker than ${\bf K^W}$ look like. Below is a table summarizing ${\bf K^W}$ and its weaker logics and the corresponding frame classes that determine them.
\[\begin{array}{|c|c|}
\hline
\text{Proof systems}&\text{Frame classes}\\
\hline
  {\bf E^W}=\texttt{PL}+\texttt{MP}+\texttt{WE}+\texttt{REW}& \text{---}, \text{ also } (n)\\
  {\bf M^W}={\bf E^W}+\texttt{WM}& (m),\text{ also }(mn)\\
  {\bf EC^W}={\bf E^W}+\texttt{WC}& (c),\text{ also }(cn)\\
  {\bf K^W}={\bf M^W}+\texttt{WC}& (mc),\text{ also \text{filters}}=(mcn)\\
\hline
\end{array}
\]

\weg{\[\begin{array}{|c|c|}
\hline
\text{Proof systems}&\text{Frame classes}\\
\hline
  {\bf E^W}=\texttt{PL}+\texttt{MP}+\texttt{WE}+\texttt{REW}& \text{---}\\
  {\bf M^W}={\bf E^W}+\texttt{WM}& (m)\\
  {\bf EC^W}={\bf E^W}+\texttt{WC}& (c)\\
  {\bf EMC^W}={\bf M^W}+\texttt{WC}& (mc)\\
  {\bf EN^W}={\bf E^W} & (n)\\
  {\bf EMN^W}={\bf M^W}& (mn)\\
  {\bf ECN^W}={\bf EC^W}& (cn)\\
  {\bf K^W}={\bf EMC^W}& \text{filters}\\
\hline
\end{array}
\]}


Note that $\texttt{WE}$ is indispensable in ${\bf K^W}$ and its weaker systems in the above table. To see this, consider an auxiliary semantics which interprets all formulas of the form $W\phi$ as $\phi$, then one may easily verify that the subsystem ${\bf K^W}-\texttt{WE}$ is sound with respect to the auxiliary semantics, but $\texttt{WE}$ is unsound, and thus $\texttt{WE}$ cannot be derived from the remaining axioms and inference rules. This entails that $\texttt{WE}$ is indispensable in ${\bf K^W}$, and thus $\texttt{WE}$ is indispensable in the weaker systems of ${\bf K^W}$.

We can also ask the following question: are all false beliefs themselves false beliefs? Different from the notion of unknown truths, the answer to this question is {\em negative}. In fact, {\em none} of false beliefs themselves are false beliefs. We now give a proof-theoretical perspective for this.
\begin{proposition}\label{prop.false-belief}
$W\phi\to \neg WW\phi$ is derivable in ${\bf E^W}$.
\end{proposition}

\begin{proof}
We have the following proof sequences:
\[
\begin{array}{lll}
(i)& WW\phi\to \neg W\phi&\texttt{WE}\\
(ii)& W\phi\to \neg WW\phi&(i),\texttt{PL}
\end{array}
\]
\end{proof}

In the reminder of this section, we focus on the completeness of the four proof systems listed above, with the aid of canonical neighborhood model constructions.  Unfortunately, all these systems may not be handled by a uniform canonical neighborhood function; rather, we need to distinguish systems excluding axiom $\texttt{WM}$ from those including it. This is similar to the case of neighborhood contingency logics~\cite{fan2019family}.

\subsubsection{Systems excluding $\texttt{WM}$}

\begin{definition}\label{def.cm-ew} Let ${\bf L}$ be a system excluding $\texttt{WM}$. A tuple $\M^{\bf L}=\lr{S^{\bf L},N^{\bf L},V^{\bf L}}$ is the {\em canonical neighborhood model} for ${\bf L}$, if
\begin{itemize}
\item $S^{\bf L}=\{s\mid s\text{ is a maximal consistent set of }{\bf L}\}$,
\item $N^{\bf L}(s)=\{|\phi|\mid W\phi\in s\}$, 
\item $V^{\bf L}(p)=\{s\in S^{\bf L}\mid p\in s\}$.
\end{itemize}
\end{definition}

The neighborhood function $N^{\bf L}$ is well defined.
\begin{lemma}
Let ${\bf L}$ be a system excluding $\texttt{WM}$. If $|\phi|=|\psi|$ and $|\phi|\in N^{\bf L}(s)$, then $W\psi\in s$.
\end{lemma}

\begin{proof}
Suppose that $|\phi|=|\psi|$ and $|\phi|\in N^{\bf L}(s)$, to prove that $W\psi\in s$. By supposition, $\vdash\phi\lra \psi$ and $W\phi\in s$. By $\texttt{REW}$, we have $\vdash W\phi\lra W\psi$, and thus $W\psi\in s$.
\end{proof}

\begin{lemma} Let ${\bf L}$ be a system excluding $\texttt{WM}$. For all $\phi\in\mathcal{L}(W)$, for all $s\in S^{\bf L}$, we have $\M^{\bf L},s\vDash \phi\iff \phi\in s$, that is, $\phi^{\M^{\bf L}}=|\phi|$.
\end{lemma}

\begin{proof}
By induction on $\phi$, where the nontrivial case is $W\phi$.

Suppose that $W\phi\in s$, to show that $\M^{\bf L},s\vDash W\phi$. By supposition and axiom $\texttt{WE}$, we derive that $\neg\phi\in s$, viz., $\phi\notin s$, then by IH, we obtain $\M^{\bf L},s\nvDash \phi$. It suffices to show that $\phi^{\M^{\bf L}}\in N^{\bf L}(s)$, which is equivalent to showing that $|\phi|\in N^{\bf L}(s)$ by IH. This follows directly from the fact that $W\phi\in s$.

Conversely, suppose that $\M^{\bf L},s\vDash W\phi$, to prove that $W\phi\in s$. By supposition and IH, $|\phi|\in N^{\bf L}(s)$ and $\phi\notin s$. This immediately gives us $W\phi\in s$.
\end{proof}

Now it is a standard work to show the following. 
\begin{theorem}\label{thm.comp-ew}
${\bf E^W}$ is sound and strongly complete with respect to the class of all neighborhood frames.
\end{theorem}


\begin{proposition}\label{thm.comp-ecw}
${\bf EC^W}$ is sound and strongly complete with respect to the class of $(c)$-frames.
\end{proposition}

\begin{proof}
For the soundness, by Thm.~\ref{thm.comp-ew}, it suffices to show the validity of $\texttt{WC}$. For this, let $\M=\lr{S,N,V}$ and $s\in S$ such that $\M,s\vDash W\phi\land W\psi$. Then $\phi^\M\in N(s)$ and $s\notin \phi^\M$, and $\psi^\M\in N(s)$ and $s\notin \psi^\M$. From $\phi^\M\in N(s)$ and $\psi^\M\in N(s)$ and $(c)$, it follows that $\phi^\M\cap \psi^\M\in N(s)$, that is, $(\phi\land\psi)^\M\in N(s)$; from $s\notin \phi^\M$ it follows that $s\notin (\phi\land\psi)^\M$. Therefore, $\M,s\vDash W(\phi\land\psi)$.

For the completeness, by Thm.~\ref{thm.comp-ew}, it is sufficient to prove that $N^{\bf L}$ has the property $(c)$. Suppose that $X\in N^{\bf L}(s)$ and $Y\in N^{\bf L}(s)$, then there are $\phi,\psi$ such that $X=|\phi|$ and $Y=|\psi|$. From $|\phi|\in N^{\bf L}(s)$ and $|\psi|\in N^{\bf L}(s)$, it follows that $W\phi\in s$ and $W\psi\in s$. By axiom $\texttt{WC}$, we obtain $W(\phi\land\psi)\in s$, thus $|\phi\land\psi|\in N^{\bf L}(s)$, namely $X\cap Y\in N^{\bf L}(s)$.
\end{proof}

\subsubsection{Systems including $\texttt{WM}$}

To deal with the completeness of the systems including $\texttt{WM}$, we need to redefine the canonical neighborhood function. The reason is as follows. If we continue using the canonical neighborhood function in Def.~\ref{def.cm-ew} and the strategy of supplementation (like the case in monotone modal logics), then we also need a rule $\dfrac{\phi\to\psi}{W\phi\to W\psi}$ in the systems. However, this rule is not sound, as one may easily check.

The following canonical neighborhood function is found to satisfy the requirement.

\begin{definition}\label{def.cm-mw}
Let ${\bf L}$ be a system including $\texttt{WM}$. A triple $\M^{\bf L}=\lr{S^{\bf L},N^{\bf L},V^{\bf L}}$ is a {\em canonical neighborhood model} for ${\bf L}$, if
\begin{itemize}
\item $S^{\bf L}=\{s\mid s\text{ is a maximal consistent set of }{\bf L}\}$,
\item $|\phi|\in N^{\bf L}(s)$ iff $W\phi\vee\phi\in s$,
\item $V^{\bf L}(p)=\{s\in S^{\bf L}\mid p\in s\}$.
\end{itemize}
\end{definition}

The reader may ask why we do not use this definition for systems excluding $\texttt{WM}$. This is because it does not work for system ${\bf EC^W}$ (Thm.~\ref{thm.comp-ecw}), as one may check.

Note that Def.~\ref{def.cm-mw} does not specify the function $N^{\bf L}$ completely; in addition to the proof sets that satisfy this definition, $N^{\bf L}$ may also contain non-proof sets relative to ${\bf L}$. Therefore, each of such logics has many canonical neighborhood models.

We need also show that $N^{\bf L}$ is well defined.

\begin{lemma}
Let ${\bf L}$ be a system including $\texttt{WM}$. If $|\phi|=|\psi|$ and $|\phi|\in N^{\bf L}(s)$, then $W\psi\vee\psi\in s$.
\end{lemma}

\begin{proof}
Suppose that $|\phi|=|\psi|$ and $|\phi|\in N^{\bf L}(s)$, to prove that $W\psi\vee\psi\in s$. By supposition, $\vdash\phi\lra \psi$ and $W\phi\vee\phi\in s$. By $\texttt{REW}$, we have $\vdash W\phi\lra W\psi$, and thus $W\psi\vee\psi\in s$.
\end{proof}

\begin{lemma} Let $\M^{\bf L}$ be a canonical neighborhood model for any system extending ${\bf M^W}$. Then for all $\phi\in\mathcal{L}(W)$, for all $s\in S^{\bf L}$, we have $\M^{\bf L},s\vDash \phi\iff \phi\in s$, that is, $\phi^{\M^{\bf L}}=|\phi|$.
\end{lemma}

\begin{proof}
By induction on $\phi$, where the nontrivial case is $W\phi$.

Suppose that $W\phi\in s$, to show that $\M^{\bf L},s\vDash W\phi$. By supposition and axiom $\texttt{WE}$, we derive that $\neg\phi\in s$, viz., $\phi\notin s$, then by IH, we obtain $\M^{\bf L},s\nvDash \phi$. It suffices to show that $\phi^{\M^{\bf L}}\in N^{\bf L}(s)$, which is equivalent to showing that $|\phi|\in N^{\bf L}(s)$ by IH. This follows directly from the fact that $W\phi\vee\phi\in s$.

Conversely, suppose that $\M^{\bf L},s\vDash W\phi$, to prove that $W\phi\in s$. By supposition and IH, $|\phi|\in N^{\bf L}(s)$ and $\phi\notin s$, which implies that $W\phi\vee\phi\in s$. Therefore, $W\phi\in s$.
\end{proof}

Given any system ${\bf L}$ extending ${\bf M^W}$, the minimal canonical neighborhood model for ${\bf L}$, denoted $\M^{\bf L}_{min}=\lr{S^{\bf L},N^{\bf L}_{min},V^{\bf L}}$, is defined such that $N^{\bf L}_{min}(s)=\{|\phi|\mid W\phi\vee\phi\in s\}$. Similar to the cases in monotone modal logic and ${\bf M^\circ}$, due to the existence of non-proof sets, the canonical neighborhood function $N^{\bf L}_{min}$ is not necessarily supplemented. So again, we use the strategy of supplementation. The notion of supplementation can be found in Def.~\ref{def.supplementation}.

\begin{theorem}\label{thm.comp-mw}
${\bf M^W}$ is sound and strongly complete with respect to the class of $(m)$-frames.
\end{theorem}

\begin{proof}
For the soundness, by Thm.~\ref{thm.comp-ew}, it remains to show the validity of $\texttt{WM}$. For this, let $\M=\lr{S,N,V}$ be an $(m)$-model and $s\in S$.

Suppose that $\M,s\vDash W(\phi\land\psi)\land\neg\psi$, to demonstrate that $\M,s\vDash W\psi$. By supposition, we have $(\phi\land\psi)^\M\in N(s)$, that is to say, $\phi^\M\cap \psi^\M\in N(s)$. Since $s\in (\neg\psi)^\M$, we have $s\notin \psi^\M$. By $(m)$, we derive that $\psi^\M\in N(s)$.  Therefore, $\M,s\vDash W\psi$.

\medskip

For the completeness, define the supplementation of $\M^{\bf L}_{min}$ and denote it $(\M^{\bf L}_{min})^+$. By definition of supplementation, $(\M^{\bf L}_{min})^+$ possesses $(m)$. Thus the remainder is to prove that $(\M^{\bf L}_{min})^+$ is indeed a canonical neighborhood model for ${\bf M^W}$. That is, for every $s\in S^c$, for every $\phi\in\mathcal{L}(W)$, we have
$$|\phi|\in (N^{\bf L}_{min})^+(s)\iff W\phi\vee\phi\in s.$$

The proof is as follows.

`$\Longleftarrow$': This follows immediately from the fact that $N^{\bf L}_{min}(s)\subseteq (N^{\bf L}_{min})^+(s)$.

`$\Longrightarrow$': Suppose that $|\phi|\in (N^{\bf L}_{min})^+(s)$, then there exists $X\in N^{\bf L}_{min}(s)$ such that $X\subseteq |\phi|$. Since $X\in N^{\bf L}_{min}(s)$, there must be a $\chi$ such that $X=|\chi|$. By $|\chi|\in N^{\bf L}_{min}(s)$, we have $W\chi\vee\chi\in s$. From $|\chi|\subseteq |\phi|$ it follows that $\vdash\chi\to\phi$. Note that the rule $\texttt{RMW}$ is derivable in ${\bf M^W}$ (Prop.~\ref{prop.derivable-rmw}). Thus an application of $\texttt{RMW}$ gives us $\vdash W\chi\land\neg\phi\to W\phi$, that is, $\vdash W\chi\to W\phi\vee\phi$. From $\vdash \chi\to\phi$ it also follows that $\vdash \chi\to W\phi\vee\phi$, and then $\vdash W\chi\vee\chi\to W\phi\vee\phi$, and therefore $W\phi\vee\phi\in s$, as required.
\end{proof}

It is shown in~\cite[Thm.~2.2, Coro.~2.7]{GilbertVenturi:2017} that ${\bf K^W}$ (denoted ${\bf A_K}$ there) is sound and complete with respect to the class of all neighborhood frames that are closed under binary intersections and are negatively supplemented, where a neighborhood frame $\mathcal{F}=\lr{S,N}$ is said to be {\em negatively supplemented} if for all $s\in S$ and $X,Y\subseteq S$, if $X\in N(s)$, $X\subseteq Y$ and $s\notin Y$, then $Y\in N(s)$. Notice that the notion of negative supplementation is weaker than that of supplementation.\footnote{For us, `weakly supplemented' seems a better term than `negatively supplemented', partly because the notion is indeed weaker than supplementation, and partly because it is not actually to negate supplementation; rather, it only adds a negative condition to the property of supplementation.} We have seen that ${\bf M^W}$ characterizes the class of neighborhood frames that are supplemented. Thus it is quite natural to ask which logic characterizes the class of neighborhood frames that are negative supplemented. As we will see, ${\bf M^W}$ does this job as well.

\begin{corollary}
${\bf M^W}$ is sound and strongly complete with respect to the class of neighborhood frames that are negatively supplemented.
\end{corollary}

\begin{proof}
The proof of the soundness is the same as in Thm.~\ref{thm.comp-mw}, by replacing $(m)$ with the property of `negative supplementation'.

The completeness also follows from Thm.~\ref{thm.comp-mw}, since negative supplementation is weaker than supplementation.
\end{proof}

We have the following conjecture. Note that the soundness is straightforward. In the current stage we do not know how to prove the completeness, because if we define $(\M^{\bf L}_{min})^+$ w.r.t. ${\bf K^W}$ as in Thm.~\ref{thm.comp-mw}, by Thm.~\ref{thm.comp-mw}, it suffices to prove that $(N^{\bf L}_{min})^+$ has $(c)$, which follows directly by Fact~\ref{fact.MtoM+} if $N^{\bf L}_{min}$ possesses $(c)$. But to show $N^{\bf L}_{min}$ possesses $(c)$, we again encounter the problem which is remarked immediately behind Def.~\ref{def.cm-mw}.\footnote{Note that we can prove the completeness based on the completeness of ${\bf K^W}$ w.r.t. the class of relational frames. Since it is shown that ${\bf K^W}$ is complete with respect to the class of relational frames~\cite{Steinsvold:falsebelief}, and each relational model has a pointwise equivalent augmented model (the proof is similar to the case in standard modal logic), and each augmented model is a filter, thus ${\bf K^W}$ is complete with respect to the class of filters, and hence also complete with respect to the class of $(mc)$-frames.}

\begin{conjecture}
${\bf K^W}$ is sound and strongly complete with respect to the class of filters, and also to the class of $(mc)$-frames.
\end{conjecture}

\weg{\begin{theorem}
${\bf K^W}$ is sound and strongly complete with respect to the class of $(mc)$-frames.
\end{theorem}

\begin{proof}
The soundness follows directly from the soundness of ${\bf M^W}$ (Thm.~\ref{thm.comp-mw}) and the validity of $\texttt{WC}$ (Thm.~\ref{thm.comp-ecw}).

For the completeness, define $(\M^{\bf L}_{min})^+$ w.r.t. ${\bf K^W}$ as in Thm.~\ref{thm.comp-mw}. By Thm.~\ref{thm.comp-mw}, it suffices to prove that $(N^{\bf L}_{min})^+$ has $(c)$. This follows from the fact that $N^c$ has $(c)$ (Thm.~\ref{thm.comp-ecw}) and the fact that the supplementation preserves $(c)$ (Fact~\ref{fact.MtoM+}).
\end{proof}

\begin{corollary}\cite[Thm.~2.2, Coro.~2.7]{GilbertVenturi:2017}
${\bf K^W}$ is sound and strongly complete with respect to the class of neighborhood frames that are closed under binary intersections and are negatively supplemented.
\end{corollary}}

\weg{\begin{theorem}
${\bf EMW}$ is sound and strongly complete with respect to the class of neighborhood frames that are supplemented and contain the unit.
\end{theorem}

\begin{proof}
Note that $S^c\in N^c(s)$ and thus $S^c\in (N^c)^+(s)$.
\end{proof}}

We close this section with a general soundness and completeness result. For those systems ${\bf L}$ including $\texttt{WM}$, as $\top\in s$, thus $W\top\vee\top\in s$, and hence $S^L=|\top|\in N^{\bf L}_{min}(s)$, then by Fact~\ref{fact.MtoM+}, we obtain that $S^{\bf L}\in (N^{\bf L}_{min})^+(s)$, which means that $(\M^{\bf L}_{min})^+$ possesses $(n)$.

However, for those systems ${\bf L}$ excluding $\texttt{WM}$, as $W\top\notin s$ (by axiom $\texttt{W1}$), by Def.~\ref{def.cm-ew}, we infer that $S^{\bf L}=|\top|\notin N^{\bf L}(s)$. Thus the canonical model $\M^{\bf L}$ for such systems ${\bf L}$ does not have $(n)$. We can handle this problem with Prop.~\ref{prop.equivalence-W}. The following general completeness result is a corollary of Prop.~\ref{prop.equivalence-W}. Note that the following result also holds for systems  including $\texttt{WM}$. 

\begin{theorem}\label{thm.generalcomp}
Let ${\bf L}$ be a system of $\mathcal{L}(W)$. If ${\bf L}$ is determined by a certain class of neighborhood frames, then it is also determined by the class of neighborhood frames satisfying $(n)$.
\end{theorem}

\begin{proof}
Suppose that ${\bf L}$ is determined by a certain class $\mathbb{C}$ of neighborhood frames, to show that ${\bf L}$ is sound and strongly complete with respect to the class of neighborhood frames satisfying $(n)$. The soundness is clear, since the class of neighborhood frames satisfying $(n)$ is contained in $\mathbb{C}$.

For the completeness, by supposition, every ${\bf L}$-consistent set, say $\Gamma$, is satisfiable in a model based on the frame in $\mathbb{C}$. That is, there exists a model $\M=\lr{S,N,V}$ where $\lr{S,N}\in\mathbb{C}$ and a state $s\in S$ such that $\M,s\vDash\Gamma$. Now, applying Prop.~\ref{prop.equivalence-W}, we obtain that $\M^{u^+},s\vDash\Gamma$ for $\M^{u^+}=\lr{S,N^{u^+},V}$, where $N^{u^+}(s)=N(s)\cup \{S\}$. Note that the definition of $N^{u^+}$ is well defined, since in Prop.~\ref{prop.equivalence-W}, $\Sigma_s$ is an arbitrary subset of $U_s$ and $S\in U_s$ (as $s\in S$) thus $\{S\}\subseteq U_s$, we can define $\Sigma_s$ to be $\{S\}$. Moreover, $\M^{u^+}$ possesses $(n)$. Also, $(n)$ does not broken the previous properties. Therefore, $\Gamma$ is also satisfiable in a neighborhood model satisfying $(n)$.
\end{proof}

\begin{corollary}
The following soundness and completeness results hold:
\begin{enumerate}
\item ${\bf E^W}$ is sound and strongly complete with respect to the class of $(n)$-frames;
\item ${\bf M^W}$ is sound and strongly complete with respect to the class of $(mn)$-frames;
\item ${\bf EC^W}$ is sound and strongly complete with respect to the class of $(cn)$-frames.
\end{enumerate}
\end{corollary}

\section{Adding public announcements}\label{sec.publicannouncements}

Now we extend the previous results to the dynamic case: public announcements. Syntactically, we add the construct $[\phi]\phi$ into the previous languages, where the formula $[\psi]\phi$ is read ``$\phi$ is the case after each truthfully public announcement of $\psi$''. Semantically, we adopt the intersection semantics proposed in~\cite{ma2013update}. In details, given a monotone neighborhood model $\M=\lr{S,N,V}$ and a state $s\in S$,
\[\begin{array}{lll}
\M,s\vDash[\psi]\phi&\iff & \M,s\vDash\psi\text{ implies }\M^{\cap\psi},s\vDash\phi\\
\end{array}\]
where $\M^{\cap\psi}$ is the intersection submodel $\M^{\cap\psi^\M}$, and the notion of intersection submodels is defined as below.


\begin{definition}\cite[Def.~3]{ma2013update}
Let $\M=\lr{S,N,V}$ be a monotone neighborhood model, and $X$ is a nonempty subset of $S$. Define the intersection submodel $\M^{\cap X}=\lr{X,N^{\cap X},V^X}$ induced from $X$, where
\begin{itemize}
\item $N^{\cap X}(s)=\{P\cap X\mid P\in N(s)\}$ for every $s\in X$,
\item $V^X(p)=V(p)\cap X$ for every $p\in\BP$.
\end{itemize}
\end{definition}

\begin{proposition}\cite[Prop.~2]{ma2013update}
The frame property $(m)$ is preserved under taking the intersection submodel. That is, if $\M$ is a monotone neighborhood model with the domain $S$, then for any $X\subseteq S$, the intersection submodel $\M^{\cap X}$ is also monotone.
\end{proposition}

We obtain the following reduction axioms for $\mathcal{L}(\bullet, W)$ and its sublanguages $\mathcal{L}(\bullet)$, $\mathcal{L}(W)$.
\[\begin{array}{lllll}
\texttt{AP} & [\psi]p\lra (\psi\to p)&&\texttt{AA}&[\psi][\chi]\phi\lra [\psi\land[\psi]\chi]\phi\\
\texttt{AN} & [\psi]\neg\phi\lra (\psi\to\neg[\psi]\phi)&&\texttt{A}\bullet&[\psi]\bullet\phi\lra (\psi\to\bullet[\psi]\phi)\\
\texttt{AC} & [\psi](\phi\land\chi)\lra ([\psi]\phi\land[\psi]\chi)&& \texttt{AW}&[\psi]W\phi\lra (\psi\to W[\psi]\phi)\\
\end{array}\]

From the reduction axioms, we can see that, every formula of $\mathcal{L}(\bullet, W)$ (and thus its sublanguages) with public announcement operators can be rewritten as a formula without public announcements via finite many of steps. Thus the addition of public announcements does not increase the expressivity of the languages in question. Moreover,

\begin{theorem}
Let $\Lambda$ be a system of $\mathcal{L}(\bullet)$ (resp. $\mathcal{L}(W)$, $\mathcal{L}(\bullet,W)$). If $\Lambda$ is sound and strongly complete with respect to the class of monotone neighborhood frames, then so is $\Lambda$ plus $\texttt{AP}$, $\texttt{AN}$, $\texttt{AC}$, $\texttt{AA}$ and $\texttt{A}\bullet$ (resp. plus $\texttt{AP}$, $\texttt{AN}$, $\texttt{AC}$, $\texttt{AA}$ and $\texttt{AW}$, plus $\texttt{AP}$, $\texttt{AN}$, $\texttt{AC}$, $\texttt{AA}$, $\texttt{A}\bullet$ and $\texttt{AW}$) under intersection semantics.
\end{theorem}

\begin{proof}
We only need to show the validity of $\texttt{A}\bullet$ and $\texttt{AW}$. The proof for the validity of other reduction axioms can be found in~\cite[Thm.~1]{ma2013update}. This then will give us the soundness. Moreover, the completeness can be shown via a standard reduction method, see~\cite{hvdetal.del:2007}. Let $\M=\lr{S,N,V}$ be any monotone neighborhood model and $s\in S$.

For $\texttt{A}\bullet$:

Suppose that $\M,s\vDash [\psi]\bullet\phi$ and $\M,s\vDash\psi$, to show that $\M,s\vDash\bullet[\psi]\phi$, that is to show $\M,s\vDash[\psi]\phi$ and $([\psi]\phi)^\M\notin N(s)$. By supposition, we have $\M^{\cap \psi},s\vDash\bullet\phi$, then $\M^{\cap\psi},s\vDash\phi$ and $\phi^{\M^{\cap\psi}}\notin N^{\cap \psi}(s)$. From $\M^{\cap\psi},s\vDash\phi$ it follows that $\M,s\vDash[\psi]\phi$. We have also $([\psi]\phi)^\M\notin N(s)$: if not, namely $([\psi]\phi)^\M\in N(s)$, then $([\psi]\phi)^\M\cap \psi^\M\in N^{\cap\psi}(s)$. Since $([\psi]\phi)^\M\cap \psi^\M\subseteq \phi^{\M^{\cap\psi}}$, by $(m)$, we derive that $\phi^{\M^{\cap\psi}}\in N^{\cap \psi}(s)$: a contradiction.

Conversely, assume that $\M,s\vDash \psi\to \bullet[\psi]\phi$, to prove that $\M,s\vDash[\psi]\bullet\phi$. For this, suppose that $\M,s\vDash\psi$, it remains to show that $\M^{\cap\psi},s\vDash\bullet\phi$, equivalently, $\M^{\cap\psi},s\vDash\phi$ and $\phi^{\M^{\cap\psi}}\notin N^{\cap\psi}(s)$. By assumption and supposition, we obtain that $\M,s\vDash\bullet[\psi]\phi$, then $\M,s\vDash[\psi]\phi$ and $([\psi]\phi)^\M\notin N(s)$. From $\M,s\vDash[\psi]\phi$ and $\M,s\vDash\psi$, it follows that $\M^{\cap\psi},s\vDash\phi$. Moreover, $\phi^{\M^{\cap\psi}}\notin N^{\cap\psi}(s)$: otherwise, $\phi^{\M^{\cap\psi}}=P\cap \psi^\M$ for some $P\in N(s)$, and then $P\subseteq (S\backslash \psi^\M)\cup\phi^{\M^{\cap\psi}}$, and thus by $(m)$, we infer that $(S\backslash \psi^\M)\cup\phi^{\M^{\cap\psi}}\in N(s)$, that is, $([\psi]\phi)^\M\in N(s)$: a contradiction.

\medskip

Now for $\texttt{AW}$:

Suppose that $\M,s\vDash[\psi]W\phi$ and $\M,s\vDash\psi$, to show that $\M,s\vDash W[\psi]\phi$, that is to show $([\psi]\phi)^\M\in N(s)$ and $\M,s\nvDash[\psi]\phi$. By supposition, we derive that $\M^{\cap\psi},s\vDash W\phi$, that is, $\phi^{\M^{\cap\psi}}\in N^{{\cap\psi}}(s)$ and $\M^{\cap\psi},s\nvDash \phi$. From $\phi^{\M^{\cap\psi}}\in N^{{\cap\psi}}(s)$, it follows that $\phi^{\M^{\cap\psi}}=P\cap \psi^\M$ for some $P\in N(s)$, and then $P\subseteq (S\backslash \psi^\M)\cup\phi^{\M^{\cap\psi}}$. By $(m)$, we get $(S\backslash \psi^\M)\cup\phi^{\M^{\cap\psi}}\in N(s)$, that is, $([\psi]\phi)^\M\in N(s)$. Moreover, from $\M,s\vDash\psi$ and $\M^{\cap\psi},s\nvDash \phi$, it follows immediately that $\M,s\nvDash[\psi]\phi$.

Conversely, assume that $\M,s\vDash\psi\to W[\psi]\phi$, to prove that $\M,s\vDash[\psi]W\phi$. For this, suppose that $\M,s\vDash\psi$, it suffices to demonstrate that $\M^{\cap\psi},s\vDash W\phi$, which means that $\phi^{\M^{\cap\psi}}\in N^{\cap\psi}(s)$ and $\M^{\cap\psi},s\nvDash\phi$. By assumption and supposition, we derive that $\M,s\vDash W[\psi]\phi$. This entails that $([\psi]\phi)^\M\in N(s)$ and $\M,s\nvDash[\psi]\phi$. From $([\psi]\phi)^\M\in N(s)$ it follows that $([\psi]\phi)^\M\cap \psi^\M\in N^{\cap\psi}(s)$. As $([\psi]\phi)^\M\cap \psi^\M\subseteq \phi^{\M^{\cap\psi}}$, by $(m)$, we gain $\phi^{\M^{\cap\psi}}\in N^{\cap\psi}(s)$. Besides, from $\M,s\nvDash[\psi]\phi$, it follows directly that $\M^{\cap\psi},s\nvDash\phi$, as desired.
\end{proof}

\weg{Similar to the case in relational semantics, we can apply the reduction axioms to Moore sentences. More precisely, we can show that Moore sentences are unsuccessful and self-refuting, and their negations are all successful.
\begin{proposition}
$[\bullet p]\neg\bullet p$ and $[\neg\bullet p]\neg\bullet p$ are provable.
\end{proposition}

\begin{proof}
Similar to the proof of~\cite[Prop.~38, Prop.~39]{Fan:2019}, because we have all necessary axioms.
\end{proof}}

For the sake of simplicity, we use ${\bf M^{\circ[\cdot]}}$ for the system that consists of ${\bf M^\circ}$ plus the above reduction axioms involving $\bullet$, and ${\bf M^{W[\cdot]}}$ for the system that consists of ${\bf M^W}$ plus the above reduction axioms involving $W$.

It is shown in~\cite[Prop.~38]{Fan:2019} that Moore sentences are unsuccessful and self-refuting, that is, $[\bullet p]\neg\bullet p$ is provable in ${\bf K^{\bullet[\cdot]}}$ (namely, the minimal Kripke logic of $\mathcal{L}(\bullet)$ plus the above reduction axioms involving $\bullet$). However, this does not apply to the monotone case.

\begin{proposition}
$[\bullet p]\neg\bullet p$ is {\em not} provable in ${\bf M^{\circ[\cdot]}}$.
\end{proposition}

\begin{proof}
We have the following proof sequences:
\[
\begin{array}{llll}
[\bullet p]\neg\bullet p&\lra & (\bullet p\to \neg[\bullet p]\bullet p)& \texttt{AN}\\
&\lra & (\bullet p\to \neg(\bullet p\to\bullet [\bullet p]p)) &\texttt{A}\bullet\\
&\lra & (\bullet p\to \neg(\bullet p\to\bullet (\bullet p\to p))) &\texttt{AP}\\
&\lra & (\bullet p\to \neg\bullet (\bullet p\to p)) & \texttt{PL}\\
\end{array}
\]

Thus we only need to show the unprovability of $\bullet p\to \neg\bullet (\bullet p\to p)$ in ${\bf M^\circ}$. By completeness of ${\bf M^\circ}$, it remains to show that this formula is not valid over the class of $(m)$-frames. To see this, just consider an $(m)$-model $\M=\lr{S,N,V}$, where $S=\{s\}$, $N(s)=\emptyset$, and $V(p)=\{s\}$. It is easy to see that $\M,s\vDash p$ and $p^\M\notin N(s)$, thus $\M,s\vDash \bullet p$. Moreover, $\M,s\vDash\bullet p\to p$ and $(\bullet p\to p)^\M\notin N(s)$, and hence $\M,s\vDash\bullet (\bullet p\to p)$, and therefore $\M,s\nvDash\bullet p\to \neg\bullet (\bullet p\to p)$. Also, $\M$ possesses $(m)$. This establishes the required result.
\end{proof}

One may show that $[\bullet p]\neg\bullet p$ is provable in ${\bf EN^\circ}$  plus the reduction axioms for $\bullet$ operator, since in ${\bf EN^\circ}$, $\bullet p\to \neg\bullet (\bullet p\to p)$ is provable, whose proof is similar as in~\cite[Prop.~38]{Fan:2019} (note that $\circ\top$ is interderivable with the rule $\dfrac{\phi}{\circ\phi}$ in the presence of the rule $\texttt{RE}\circ$).

Similar to the case in the minimal Kripke logic for $\mathcal{L}(\bullet)$, in ${\bf M^{\circ[\cdot]}}$, the negations of Moore sentences are all successful formulas.
\begin{proposition}
$[\neg\bullet p]\neg \bullet p$ is provable in ${\bf M^{\circ[\cdot]}}$.
\end{proposition}

\begin{proof}
The proof is similar to that of~\cite[Prop.~39]{Fan:2019} except that we are now in the much weaker system. In this system, we have the following proof sequences:
\[
\begin{array}{llll}
[\neg\bullet p]\neg \bullet p&\lra &(\neg \bullet p\to \neg[\neg\bullet p]\bullet p)&\texttt{AN}\\
&\lra & (\neg \bullet p\to \neg (\neg \bullet p\to \bullet [\neg \bullet p]p)) &\texttt{A}\bullet\\
&\lra & (\neg \bullet p\to \neg (\neg \bullet p\to \bullet (\neg \bullet p\to p))) & \texttt{AP}\\
&\lra & (\neg \bullet p\to \neg \bullet (\neg \bullet p\to p)) & \texttt{PL}\\
&\lra & (\circ p\to \circ(\circ p\to p)) &\text{Def.~of~}\circ\\
\end{array}
\]

Notice that $\circ p\to \circ(\circ p\to p)$ is provable in ${\bf M^\circ}$. First, as $\vdash p\to (\circ p\to p)$, by rule $\texttt{RM}\circ$ (Prop.~\ref{prop.derivable-rm}), $\vdash \circ p\land p\to\circ(\circ p\to p)$. Moreover, $\vdash \circ p\land \neg p\to\circ(\circ p\to p)$: to see this, we consider its contraposition, that is, $\bullet (\circ p\to p)\to (\circ p\to p)$, which is just an instance of axiom $\circ\texttt{E}$.
\end{proof}

Interestingly, public announcements cannot change one's false belief about a fact. More precisely, if you have a false belief about $p$ and someone responds with ``you are wrong about $p$'', then you still have the false belief.
\begin{proposition}
$[Wp]Wp$ is provable in ${\bf M^{W[\cdot]}}$.
\end{proposition}

\begin{proof}
We observe the following proof sequences:
\[
\begin{array}{llll}
[Wp]Wp&\lra & (Wp\to W[Wp]p)&\texttt{AW}\\
&\lra & (Wp\to W(Wp\to p))&\texttt{AP}\\
\end{array}
\]

Moreover, $Wp\to W(Wp\to p)$ is provable in ${\bf M^W}$. To see this, note that $\vdash p\to (Wp\to p)$, then by rule $\texttt{RMW}$ (Prop.~\ref{prop.derivable-rmw}), we derive that $\vdash Wp\land\neg (Wp\to p)\to W(Wp\to p)$, that is, $\vdash Wp \land Wp\land\neg p\to W(Wp\to p)$. Now by $\texttt{WE}$, we obtain that $\vdash Wp\to W(Wp\to p)$.
\end{proof}


\section{Conclusion and Future work}\label{sec.conclusion}

In this paper, we investigated logics of unknown truths and false beliefs under neighborhood semantics. More precisely, we compared the relative expressivity of the two logics, proposed notions of $\bullet$-morphisms and $W$-morphisms with applications to frame definability, a general soundness and completeness result and some related results in the literature in a relative easy way, and axiomatized the two logics over various neighborhood frames, and finally, we extended the results to the case of public announcements, where by adopting the intersection semantics we found suitable reduction axioms and thus complete proof systems, which again has good applications to Moore sentence and some others.

An interesting question is to explore the notions of bisimulations for logics of unknown truths and false beliefs, for which notions of $\bullet$-morphisms and $W$-morphisms might give us some inspirations. Moreover, a related research direction would be neighborhood bimodal logics with contingency and accident.


\bibliographystyle{plain}
\bibliography{biblio2019}
\end{document}